\documentclass[11pt]{amsart}

\usepackage{amsmath,amssymb,latexsym,bm,xcolor,hyperref}%,pb-diagram,pb-xy,graphicx}
\usepackage[all, arc, cmtip]{xy}
\newcommand{\bC}{\mbox{${\mathbb C}$}}

\def\cT{\mathcal{T}}

\def\mytilde{\kern-.015in\hbox{\lower.03in\hbox{\~{}}}\kern-.01in}

\DeclareMathOperator{\Hom}{Hom}
\DeclareMathOperator{\RHom}{RHom}
\DeclareMathOperator{\pt}{pt}
\DeclareMathOperator{\Spec}{Spec}

\DeclareMathOperator{\MC}{MC}   % motivic Chern class in K theory
\DeclareMathOperator{\SMC}{SMC}   % Segre motivic Chern class in K theory
\DeclareMathOperator{\MHC}{MHC}
\DeclareMathOperator{\MHM}{MHM}
\DeclareMathOperator{\IC}{IC}
\DeclareMathOperator{\KL}{KL}    % KL Schubert class in hyperbolic
\newcommand{\kl}{C}         %KL class in K theory
\newcommand{\klt}{\wt C}    % KL widetilde class in K theory
\newcommand{\wt}{\widetilde}
\newcommand{\hiota}{\hat \iota}  % the involution preserving \tau

\newcommand{\calJ}{{\cal J}}
\newcommand{\calN}{{\mathcal{N}}}
\newcommand{\calO}{{\mathcal{O}}}

\newcommand{\calD}{{\mathcal D}}

\newtheorem{thm}{Theorem}

\newtheorem{prop}[thm]{Proposition}

\newtheorem{dfn}[thm]{Definition}

\newtheorem{cor}[thm]{Corollary}
\newtheorem{lem}[thm]{Lemma}
\newtheorem{exa}[thm]{Example}

\newtheorem{rema}[thm]{Remark}

\newcommand{\nc}{\newcommand}
\nc{\rnc}{\renewcommand}
\nc{\bb}[1]{{\mathbb #1}}
\nc{\bbA}{\bb{A}}\nc{\bbB}{\bb{B}}\nc{\bbC}{\bb{C}}\nc{\bbD}{\bb{D}}
\nc{\bbE}{\bb{E}}\nc{\bbF}{\bb{F}}\nc{\bbG}{\bb{G}}\nc{\bbH}{\bb{H}}
\nc{\bbI}{\bb{I}}\nc{\bbJ}{\bb{J}}\nc{\bbK}{\bb{K}}\nc{\bbL}{\bb{L}}
\nc{\bbM}{\bb{M}}\nc{\bbN}{\bb{N}}\nc{\bbO}{\bb{O}}\nc{\bbP}{\bb{P}}
\nc{\bbQ}{\bb{Q}}\nc{\bbR}{\bb{R}}\nc{\bbS}{\bb{S}}\nc{\bbT}{\bb{T}}
\nc{\bbU}{\bb{U}}\nc{\bbV}{\bb{V}}\nc{\bbW}{\bb{W}}\nc{\bbX}{\bb{X}}
\nc{\bbY}{\bb{Y}}\nc{\bbZ}{\bb{Z}}
\nc{\mbf}[1]{{\mathbf #1}}
\nc{\bfA}{\mbf{A}}\nc{\bfB}{\mbf{B}}\nc{\bfC}{\mbf{C}}\nc{\bfD}{\mbf{D}}
\nc{\bfE}{\mbf{E}}\nc{\bfF}{\mbf{F}}\nc{\bfG}{\mbf{G}}\nc{\bfH}{\mbf{H}}
\nc{\bfI}{\mbf{I}}\nc{\bfJ}{\mbf{J}}\nc{\bfK}{\mbf{K}}\nc{\bfL}{\mbf{L}}
\nc{\bfM}{\mbf{M}}\nc{\bfN}{\mbf{N}}\nc{\bfO}{\mbf{O}}\nc{\bfP}{\mbf{P}}
\nc{\bfQ}{\mbf{Q}}\nc{\bfR}{\mbf{R}}\nc{\bfS}{\mbf{S}}\nc{\bfT}{\mbf{T}}
\nc{\bfU}{\mbf{U}}\nc{\bfV}{\mbf{V}}\nc{\bfW}{\mbf{W}}\nc{\bfX}{\mbf{X}}
\nc{\bfY}{\mbf{Y}}\nc{\bfZ}{\mbf{Z}}
\nc{\bfa}{\mbf{a}}\nc{\bfb}{\mbf{b}}\nc{\bfc}{\mbf{c}}\nc{\bfd}{\mbf{d}}
\nc{\bfe}{\mbf{e}}\nc{\bff}{\mbf{f}}\nc{\bfg}{\mbf{g}}\nc{\bfh}{\mbf{h}}
\nc{\bfi}{\mbf{i}}\nc{\bfj}{\mbf{j}}\nc{\bfk}{\mbf{k}}\nc{\bfl}{\mbf{l}}
\nc{\bfm}{\mbf{m}}\nc{\bfn}{\mbf{n}}\nc{\bfo}{\mbf{o}}\nc{\bfp}{\mbf{p}}
\nc{\bfq}{\mbf{q}}\nc{\bfr}{\mbf{r}}\nc{\bfs}{\mbf{s}}\nc{\bft}{\mbf{t}}
\nc{\bfu}{\mbf{u}}\nc{\bfv}{\mbf{v}}\nc{\bfw}{\mbf{w}}\nc{\bfx}{\mbf{x}}
\nc{\bfy}{\mbf{y}}\nc{\bfz}{\mbf{z}}

\newcommand{\al}{\alpha}

\newcommand{\de}{\delta}
\newcommand{\ep}{\epsilon}
\newcommand{\la}{\lambda}
\newcommand{\fe}{\mathfrak{e}}
\newcommand{\fh}{\mathfrak{h}}

\newcommand{\unit}{\mathbf{1}}
\newlength{\cellsize}
\newcommand\tableau[1]{
\vcenter{
\let\\=\cr
\baselineskip=-16000pt
\lineskiplimit=16000pt
\lineskip=0pt
\halign{&\tableaucell{##}\cr#1\crcr}}}

\newcommand{\tableaucell}[1]{{%
\def \arg{#1}\def \void{}%
\ifx \void \arg
\vbox to \cellsize{\vfil \hrule width \cellsize height 0pt}%
\else
\unitlength=\cellsize
\begin{picture}(1,1)
\put(0,0){\makebox(1,1){$#1$}}
\put(0,0){\line(1,0){1}}
\put(0,1){\line(1,0){1}}
\put(0,0){\line(0,1){1}}
\put(1,0){\line(0,1){1}}
\end{picture}%
\fi}}

\makeatletter

\let\choose\@@choose
\let\cal\mathcal
\makeatother

\usepackage[left=1in, right=1in, top=1in, bottom=1in]{geometry}

\title[Geometric properties of the Kazhdan-Lusztig Schubert basis]{Geometric properties of \\the Kazhdan-Lusztig Schubert basis}
\author[C.~Lenart]{Cristian~Lenart}
\address{State University of New York at Albany, 1400 Washington Avenue, Albany, NY 12222}
\email{clenart@albany.edu}

\author[C.~Su]{Changjian~Su}
\address{University of Toronto, 40 St. George St., Toronto, ON M5S 2E4, Canada}
\email{csu@math.toronto.edu}

\author[K.~Zainoulline]{Kirill~Zainoulline}
\address{Department of Mathematics and Statistics, University of Ottawa, 150 Louis-Pasteur, 
Ottawa,  ON,  K1N 6N5, 
Canada}
\email{kirill@uottawa.ca}

\author[C.~Zhong]{Changlong~Zhong}
\address{State University of New York at Albany, 1400 Washington Avenue, Albany, NY 12222}
\email{czhong@albany.edu}
\begin{document}

\begin{abstract}
We study classes determined by the Kazhdan-Lusztig basis of the Hecke algebra in the $K$-theory and hyperbolic cohomology theory of flag varieties. We first show that, in $K$-theory,  the two different choices of Kazhdan-Lusztig bases produce dual bases, one of which can be interpreted as characteristic classes of the intersection homology mixed Hodge modules. In equivariant hyperbolic cohomology, we show that  if the Schubert variety is smooth, then the class it determines coincides with the class of the Kazhdan-Lusztig basis; this was known as the Smoothness Conjecture. 
%, which proves the Smoothness Conjecture of \cite{LZZ}. 
For Grassmannians, we prove that the classes of the Kazhdan-Lusztig basis coincide with the classes determined by Zelevinsky's small resolutions. These properties of the so-called KL-Schubert basis show that it is the closest existing analogue to the Schubert basis for hyperbolic cohomology; the latter is a very useful testbed for more general elliptic cohomologies.

\end{abstract}

\maketitle

\tableofcontents

\section{Introduction} Let $G$ be a split semi-simple linear algebraic group with a fixed Borel subgroup $B$ and  a maximal torus $T\subset B$. Let $P$ be a parabolic subgroup containing the Borel subgroup $B$. The varieties $G/P$ and $G/B$ are called flag varieties, and they are among the most concrete objects in algebraic geometry, because of the Bruhat decompositions. For instance, the equivariant cohomology (Chow group) of flag varieties  is freely spanned by the classes of Schubert varieties $X(w)$. Similarly, the equivariant $K$-theory of flag varieties  is spanned by the structure sheaves of Schubert varieties. The field of studying intersection theory of these classes is called Schubert calculus, and is related to combinatorics, representation theory, and enumerative geometry.

Due to the failure of Schubert varieties being smooth, the present paper deals with two different directions in generalizing classical Schubert calculus. The first one is concerned with the Chern classes. Although the classical Chern class theory does not work for the singular Schubert varieties, there are generalizations to this case, which are called Chern-Schwartz-MacPherson (CSM) classes \cite{M74,  S65a, S65b} in homology and motivic Chern (MC) classes in $K$-theory \cite{BSY10, AMSS19, FRW18}. These generalized Chern classes of Schubert cells are closely related to the corresponding stable bases of the cotangent bundle $T^*G/B$, defined by Maulik and Okounkov in their study of quantum cohomology/$K$-theory of Nakajima quiver varieties \cite{MO19, O17}. These classes are permuted by various Demazure-Lusztig operators \cite{AM16, Su17, SZZ17, AMSS19, MNS}, and are related to unramified principal series representations of the Langlands dual group over a non-archimedean local field \cite{SZZ17, AMSS19}.

We focus on the Kazhdan-Lusztig bases of the Hecke algebra, which are related to the intersection cohomology of Schubert varieties. Classically, there are two choices of Kazhdan-Lusztig bases. In this paper, we consider the $K$-theory classes determined by these two collections of Kazhdan-Lusztig bases. The cohomology case is considered in \cite{MS20}. As our first main result, we show that they are dual to each other in Theorem~\ref{thm:dual} and \ref{thm:KLdualP}. These dualities are closely related to the characteristic classes of mixed Hodge modules, studied by Sch\"urmann and his collaborators \cite{S11, S17, BSY10}. Moreover, we interpret one collection of these classes as the motivic Hodge Chern classes of the intersection homology mixed Hodge modules of the Schubert varieties, which immediately implies that they are invariant under the Serre-Grothendieck duality, see Proposition~\ref{prop:cwIC} and Corollary~\ref{cor:inv}.

The other direction is to look at more general cohomology theories, namely the equivariant oriented cohomology theories of Levine-Morel. They are those contravariant  functors $\bfh_T$ from the category of smooth (quasi)-projective varieties to the category of commutative rings, such that for any proper map of varieties, a push-forward of the cohomology groups is defined. One can then define Chern classes, where the first Chern class of the tensor product of line bundles determines a one-dimensional commutative formal group law. The structure of the equivariant oriented cohomology of flag varieties is studied in \cite{CZZ1, CZZ2, CZZ3, LZZ}. Roughly speaking, there is an algebra generated by push-pull operators between $\bfh_T(G/B)$ and $\bfh_T(G/P)$, called the formal affine Demazure algebra $\bfD_F$, whose dual $\bfD_F^*$ is isomorphic to $\bfh_T(G/B)$. 

To resolve the singularities of a Schubert variety $X(w)$, one often uses the Bott-Samelson resolution, which is defined by fixing a reduced decomposition of the Weyl group element $w$.  For oriented cohomology beyond singular cohomology/$K$-theory, the classes determined by such resolutions depend on the choice of the reduced decomposition. This corresponds to the fact that, for general $\bfh_T$, the push-pull operators do not satisfy the braid relations. Because of this fact, there are no canonically defined Schubert classes. 

Aiming for the definition of Schubert classes, in \cite{LZ17, LZZ}, the authors consider the so-called hyperbolic cohomology, denoted by $\fh$. A Riemann-Roch type map is defined from $K$-theory to the hyperbolic cohomology theory, which induces an action of the Hecke algebra (considered on the $K$-theory side) on the hyperbolic cohomology of $G/B$. In this way, the action of the  Kazhdan-Lusztig basis defines classes $\KL_w$ in $\fh_T(G/B)$, called KL-Schubert classes. In \cite{LZ17, LZZ}, there is a conjecture  stating that,  if the Schubert variety $X(w)$ is smooth, then its fundamental class coincides with the class $\KL_w$. It is proved in some special cases in \cite{LZ17, LZZ}.  The second main result of this paper is to prove this conjecture in full generality, see Theorem~\ref{thm:main}.

The idea of the proof is as follows: if $X(w)$ is smooth, then all the Kazhdan-Lusztig polynomials $P_{y,w}$ for any $y\leq w$ are equal to 1, so the Kazhdan-Lusztig basis for $w$ is the sum of the Demazure-Lusztig operators. As mentioned above, the MC classes of Schubert cells in $K$-theory are permuted by the Demazure-Lusztig operators. So the MC class of $X(w)$ coincides with the KL class in $K$-theory, and the restriction formula for the former is obtained in \cite{AMSS19} by generalizing a result of Kumar~\cite{K96}. Translating this formula to the hyperbolic side, we prove the Smoothness Conjecture (Theorem \ref{thm:main}). For partial flag varieties, a similar property is also proved. 

Restricting to type $A$ Grassmannians, we prove more geometric and combinatorial properties. For example, Zelevinsky constructed small resolutions of all Schubert varieties~\cite{Z83}. We prove that the classes determined by these resolutions coincide with the KL-Schubert classes, which  is the third main result of this paper, see Theorem~\ref{kl-zel}. By the uniqueness of the Kazhdan-Lusztig basis, it follows that all small resolution classes are the same. The proof of this result can be summarized as follows: Zelevinsky's small resolutions are similar to the Bott-Samelson resolutions, except that, instead of using minimal parabolic subgroups, one considers more general parabolic subgroups. So the small resolution classes can be computed by using relative push-pull operators between $G/P$ and $G/Q$. These operators were studied in~\cite{CZZ2}. On the other hand, in \cite{KL}, a factorization of the Kazhdan-Lusztig basis elements for Grassmannians is exhibited. By carefully transforming this factorization, one can write the Kazhdan-Lusztig basis elements as products of ``relative'' Kazhdan-Lusztig elements. Finally, by identifying the latter with the relative push-pull operators, one proves Theorem~\ref{kl-zel}. 

There have been important developments in Schubert calculus for general cohomology theories. More specifically, for elliptic cohomology, a stable basis in the cotangent bundle $T^*G/B$ was defined (see \cite{AO16,O20}, which generalizes stable bases for cohomology and $K$-theory),  and canonical classes were associated with Bott-Samelson resolutions of Schubert varieties~\cite{RW19,  KRW20}. The elliptic cohomology used in the latter papers can be considered as the oriented cohomology theory associated with a certain elliptic formal group law determined by the Jacobi theta functions; meanwhile, the mentioned cohomology classes are elliptic analogues of the CSM classes in ordinary cohomology and the MC classes in $K$-theory. On the other hand, the hyperbolic formal group law we consider here comes from a generic singular Weierstrass curve, see~\cite{BB}. The properties of the KL-Schubert basis proved in this paper (namely, the Smoothness Conjecture and the interpretation in terms of the Zelevinsky small resolutions) show that this basis is the closest existing analogue to the Schubert basis for hyperbolic cohomology. Furthermore, the latter is a very useful testbed for more general elliptic cohomologies.

The paper is organized as follows. In Section~2, we recall the algebraic construction of the equivariant oriented cohomology of flag varieties. In Section~3, we recall basic facts about the Hecke algebra, MC classes, and the smoothness criterion. In Section~4, we use Kazhdan-Lusztig bases to define the two collections of  KL classes in $K_T(G/B)$ and $K_T(G/P)$, and show that they are dual to each other. We also give a geometric interpretation for one of them using mixed Hodge modules. In Section~5, we recall the definition of KL-Schubert classes in hyperbolic cohomology, and prove the Smoothness Conjecture. In Section~6, we prove Theorem~\ref{kl-zel}, which connects small resolutions for Grassmannians with the corresponding KL-Schubert classes. 

\noindent\textbf{Acknowledgments:} We would like to thank Samuel Evens for helpful conversations. C.~L. gratefully acknowledges the partial support from the NSF grants DMS-1362627 and DMS-1855592.  K.~Z.  acknowledges the partial support from the NSERC Discovery grant RGPIN-2015-04469, Canada.  C.~S. thanks J. Sch\"urmann for useful discussions, and further to P. Aluffi, L. Mihalcea, H. Naruse and G. Zhao for related collaborations. 

\section{Formal affine Demazure algebra and its dual}\label{sec:DEM}
We recall the definition of the formal affine Demazure algebra and its relation with equivariant generalized (oriented) cohomology of flag varieties following \cite{HMSZ,CZZ1,CZZ2} and especially the paper \cite{CZZ3}. 

\subsubsection*{Notation}
Let $G$ be a  semisimple simply connected linear algebraic 
group over $\mathbb{C}$, and fix $B$ a Borel subgroup with a maximal torus $T\subset B$. Let 
$X^*(T)$ denote the character lattice of $T$.  Let $W=N_G(T)/T$ be the Weyl group.

Let $\Sigma$ denote the set of associated roots and let $\Sigma^+$ 
denote the subset of roots in $B$. For any root $\alpha$, let $\alpha>0$ (resp. 
$\alpha<0$) denote $\alpha\in \Sigma^+$ (resp. $-\alpha\in \Sigma^+$).

Let $\Pi=\{\alpha_1,\ldots,\alpha_n\}$ denote the set of simple roots. Let $\ell \colon W \to \mathbb{Z}$ denote the length function.  For any 
$J\subset \Pi$, denote by $W_J$ the parabolic subgroup corresponding to $J$, by $w_J$ its longest element, and by $W^J$ the set of minimal length representatives of right cosets $W/W_J$. Specifically, $w_0:=w_\Pi\in W$ is the longest element. More generally, if $J'\subset J\subset \Pi$, denote $w_{J/J'}:=w_Jw_{J'}\in W^{J'}$, that is, $w_{J/J'}$ is the maximal element  (in terms of the Bruhat order) in the set $W_J\cap W^{J'}$. Denote $\Sigma_J:=\{\al\in \Sigma|s_\al\in W_J\}$, and $\Sigma_J^\pm:=\Sigma_J\cap \Sigma^\pm$. 

\subsubsection*{Formal group algebra}
Let $F$ be a one dimensional formal group law over a  commutative unital ring $R$.  The formal group algebra $R[[X^*(T)]]_F$ is defined to be the quotient of the completion
\[
R[[x_\la|\la\in X^*(T)]]/\calJ_F\,,
\]
where $\calJ_F$ is the closure of the ideal generated by $\langle x_0, F(x_\la,x_\mu)-x_{\la+\mu}  \mid \la,\mu\in X^*(T)\rangle$. For simplicity it will be denoted by $S$. 
It can be shown that if $\{\omega_1,...,\omega_n\}$ is a basis of $X^*(T)$, then $S$ is (non-canonically) isomorphic to $R[[\omega_1,\ldots,\omega_n]]$. 

\subsubsection*{Localized twisted group ring}
Let $Q=S\left[\frac{1}{x_\al}|\al>0\right]$, and $Q_W=Q\otimes_R R[W]$. Denote the canonical left $Q$-basis of $Q_W$ by $\de_w, w\in W$, and define a product on $Q_W$ by 
\[
(p\de_w)\cdot (p'\de_{w'}):=pw(p')\de_{ww'}\,, \quad p,p'\in Q, w,w'\in W\,.
\]
In particular, we have $\de_vp=v(p)\de_v$, $p\in Q$.

\subsubsection*{Push-pull elements}
For each root $\al$, define the formal push-pull element 
\[
Y_\al:=(1+\de_{s_\al})\tfrac{1}{x_{-\al}}\in Q_W\,.
\]
For any reduced word $w=s_{i_1}\cdots s_{i_k}$, where $s_i$ is the simple reflection corresponding to the $i$th simple root in $\Pi$,
define $I_w=(i_1,\ldots,i_k)$, and $Y_{I_w}=Y_{\al_{i_1}}\cdots Y_{\al_{i_k}}$. The product $Y_{I_w}$ depends on the choice of the reduced sequence, unless the formal group law $F$ is of the form $x+y+\beta xy$ with $\beta\in R$. For simplicity, denote $\de_i:=\de_{s_i}$, $Y_i:=Y_{\al_i}$ and $x_{\pm i}:=x_{\pm \al_i}$. 

\subsubsection*{Formal affine Demazure algebra}
Let $\bfD_F$ be the subring of $Q_W$ generated by elements of $S$ and push-pull elements $Y_i$, $i=1,\ldots, n$. This is called the formal affine Demazure algebra. It is proved in \cite{CZZ1} that $\bfD_F$ is a free left $S$-module with basis $\{Y_{I_w}|w\in W\}$. 

\begin{exa}{\rm 
If $R=\bbZ$ and $F_m(x,y)=x+y-xy$  (multiplicative formal group law), then  \[S\cong \bbZ[X^*(T)]^\wedge\,,\qquad x_\al \mapsto1-e^{-\al}\,,\] where the completion is taken with respect to the kernel of the trace map.  The ring $\bfD_F$ is then isomorphic to the (completed) affine $0$-Hecke algebra.}
\end{exa}

For $J'\subset J\subseteq \Pi$, denote
\[
x_{J/J'}:=\prod_{\al\in \Sigma^-_J\backslash \Sigma^-_{J'}}x_\al\,, \quad x_J:=x_{J/\emptyset}\,. 
\]
Fixing a set of left coset representatives $W_{J/J'}$ of $W_J/W_{J'}$, we define a push-pull element 
\begin{equation}\label{eq:relpush}
Y_{J/J'}:=\left(\sum_{w\in W_{J/J'}}\de_w\right)\frac{1}{x_{J/J'}}\,, \quad Y_J:=Y_{J/\emptyset}=\left(\sum_{w\in W_J}\de_w\right)\frac{1}{x_J}\,. 
\end{equation}
Note that the definition of $Y_{J/J'}$ does not depend on the choice of $W_{J/J'}$. 
If $J=\Pi$, $x_\Pi$ and $Y_\Pi$ are correspondingly defined. 
For instance, if $J=\{i\}$, then $Y_{\{i\}}=Y_{\al_i}$. Note that in general $Y_{J/J'}\in Q_W$, but $Y_J\in \bfD_F$. We have
\begin{equation}\label{mult-y}
Y_{J/J'}Y_{J'}=Y_J\,.
\end{equation}

There is an anti-involution $\iota$ of $\bfD_F$, defined by 
\begin{equation}\label{eq:inv1}
\iota(p\de_v):=\de_{v^{-1}}p\frac{v(x_\Pi)}{x_\Pi}=v^{-1}(p)\frac{x_\Pi}{v^{-1}(x_\Pi)}\de_v\,, \quad p\in Q, \;v\in W\,.
\end{equation}
By definition, if $I^{-1}$ is the sequence obtained from $I$ by reversing the order, then 
\begin{equation} \label{invar-yj} 
\iota (Y_{I})=Y_{I^{-1}}\,.
\end{equation}

\subsubsection*{Dual of the Demazure algebra}
Let $\bfD_F^*$  denote the $S$-linear dual $\Hom_S(\bfD_F, S)$ with the dual basis $Y_{I_w}^*$, $w\in W$. One can also consider the $Q$-linear dual $Q_W^*=\Hom_Q(Q_W, Q)$, which is isomorphic to the set-theoretic $\Hom(W,Q)$. There is the dual basis $f_w, w\in W$ of $Q_W^*$ such that $f_w(\de_v)=\de_{w,v}^{Kr}$ and
$
f_w\cdot f_v=\de_{w,v}^{Kr}f_w$. 
It turns $Q_W^*$ into a commutative ring with identity $\unit=\sum_wf_w$. 
By definition, we have $\bfD_F^*\subset Q_W^*$ (where the former is a $S$-module, and the latter is considered as a $Q$-module), and the product on $Q_W^*$ restricts to the product on $\bfD_F^*$.

\subsubsection*{Two actions on the dual}
There are actions denoted `$\bullet$' and `$\odot$' of the ring $Q_W$ on its $Q$-linear dual $Q_W^*$ defined as:
\begin{equation}\label{eq:act}
(p\de_v)\bullet (qf_w):=qwv^{-1}(p)f_{wv^{-1}} \quad\text{ and }\quad (p\de_v)\odot (qf_w):=pv(q)f_{vw}\,, \quad v,w\in W,\; p,q\in Q\,.
\end{equation}
It follows from \cite[\S3]{LZZ} that the $\bullet$-action is $Q$-linear, while the $\odot$-action is not, and the two actions commute. We also have $z\bullet \pt_e=\iota(z)\odot \pt_e$. Moreover,  the two actions induce (via the embeddings $\bfD_F\subset Q_W$ and $\bfD_F^*\subset Q_W^*$) corresponding actions of $\bfD_F$ on $\bfD_F^*$. For homology and $K$-theory, the $\bullet$ and $\odot$ actions correspond to the right and left actions considered in \cite{MNS}. 

%\begin{lem}\label{lem:inv}We have 
%\[
%Y_\Pi\bullet ((z\bullet f)\cdot g)=Y_\Pi\bullet [f\cdot (\iota(z)\bullet g)]\,, \quad z\in Q_W\,, \quad f,g\in Q_{W}^*\,.\]
%\end{lem}
%\begin{proof}Since the $\bullet$-action is $Q$-linear, it suffices to assume that $f=f_w, g=f_v$ and $z=p\de_u$ with $p\in Q, u,v,w\in W$. The conclusion then follows from the definition of the $\bullet$-action and of the element $Y_\Pi$. 
%\end{proof}

\subsubsection*{The class of a point}
For each $w\in W$ define the element 
\[
\quad \pt_w:=x_\Pi \bullet f_w= w(x_\Pi)f_w\,,
\]
and call it the class of a point.
From the definition,  we have $z\bullet \pt_e=\iota(z)\odot \pt_e$, $z\in Q_W$,  where $e\in W$ denotes the identity element. 

\subsubsection*{Generalized (oriented) cohomology} Given a formal group law $F$ over $R$,
let $\bfh$ be the corresponding free algebraic generalized (oriented) cohomology theory 
obtained from the algebraic cobordism $\Omega$ of Levine-Morel~\cite{LM07} by tensoring with $F$, i.e. \[\bfh(-):=\Omega(-)\otimes_{\Omega(\pt)} R\,.\]
Here $\Omega(pt)$ is the Lazard ring, the coefficient ring of universal formal group law, and $\Omega(\pt)\to R$ is the evaluation map defining $F$. We refer to~\cite{LM07}
for all the properties of such theories.

In particular, for the additive formal group law $F_a(x,y)=x+y$ one obtains the Chow ring and for the multiplicative group law $F_m$
one gets the usual $K$-theory.

\subsubsection*{Equivariant generalized cohomology}
Let $\bfh_T$ be the respective $T$-equivariant generalized (oriented) cohomology theory of \cite[\S2]{CZZ3}. Replacing $\bfh_T$ if necessary by its characteristic
completion (see \cite[\S3]{CZZ3}), the main result of \cite{CZZ3} says that
the formal affine Demazure algebra $\bfD_F$ and its dual $\bfD_F^*$ are related to generalized cohomology $\bfh_T(G/B)$ and $\bfh_T(G/P_J)$ as follows:
\begin{enumerate}
\item There is an isomorphism $\bfD_F^*\cong \bfh_T(G/B)$, which maps the element $Y_{I_w^{-1}}\bullet \pt_e=Y_{I_w}\odot \pt_e$ to the Bott-Samelson class determined by the sequence $I_w$.
\item Via the above isomorphism, the map $Y_\Pi\bullet\_:\bfD_F^*\to (\bfD_F^*)^W\cong S$ coincides with the map $\bfh_T(G/B)\to \bfh_T(\Spec(k))$. 
\item The group $W$ acts on $\bfD^*_F$ by restriction of the $\bullet$-action via the embedding $W\subset \bfD_F$.  For any subset $J\subset \Pi$, one has an isomorphism $(\bfD_F^*)^{W_J}\cong \bfh_T(G/P_J)$, and the map $Y_J:\bfD_F^*\to (\bfD_F^*)^{W_J}$ coincides with the push-forward map $\bfh_T(G/B)\to \bfh_T(G/P_J)$. More generally,   the map $Y_{J/J'}\bullet \_:Q_W^*\to Q_W^*$ restricts to a map $(\bfD_F^*)^{W_{J'}}\to (\bfD_F^*)^{W_J}$, which corresponds to $\bfh_T(G/P_{J'})\to \bfh_T(G/P_J) $.
\item The embedding $\bfD_F^*\to Q_W^*$ coincides with the restriction to $T$-fixed points map $\bfh_T(G/B)\to Q\otimes_S\bfh_T(W) $, and the element $\pt_w$ is mapped to the class $\fe_w$ of $T$-fixed point  of $G/B$. 
\end{enumerate}

\begin{rema}{\rm 
Observe that the localization axiom \cite[A3]{CZZ3} used to prove the above properties
can be replaced by a weaker CD-property of \cite[Def.~3.3]{NPSZ} which holds for any $\bfh_T$ defined using the Borel construction (see \cite[Example~3.6]{NPSZ}).}
\end{rema}

\section{Hecke algebra, motivic Chern class, and the smoothness criterion }\label{dem-lus-kls}
In this section, we recall the definition of the Kazhdan-Lusztig Schubert (KL-Schubert) classes, following \cite{LZZ}.

\subsubsection*{The multiplicative case}
Set $R=\bbZ[t,t^{-1}, (t+t^{-1})^{-1}]$, where $t$ is a parameter. 
Definitions of section~\ref{sec:DEM} applied to the multiplicative formal group law $F_m$ over $R$ give the respective formal group algebra and its localization:\[
S_m:=R[[X^*(T)]]_{F_m}\,,\qquad  Q_m:=S_{m}\left[\tfrac{1}{x_\al}|\al>0\right]\,;\]
the localized twisted group algebra and the formal affine Demazure algebra:
\[Q_{m,W}:=Q_m\otimes_R R[W]\,,\qquad \bfD_m:=\langle S_m,Y_1,\ldots, Y_n\rangle \subset Q_{m, W}\,.\]  

\subsubsection*{The Demazure-Lusztig elements}
Define the Demazure-Lusztig elements in $Q_{m,W}$ as
\[\tau_i:=Y^m_i(t-t^{-1}e^{\al_i})-t=\frac{t^{-1}-t}{1-e^{-\al_i}}+\frac{t-t^{-1}e^{-\al_i}}{1-e^{-\al_i}}\de_i^m\,.\]
%\tau_i^\vee:=(t-t^{-1}e^{\alpha_i})Y^m_{i}-t=\frac{t^{-1}-t}{1-e^{-\alpha_i}}+\frac{t-t^{-1}e^{\alpha_i}}{1-e^{-\alpha_i}}\delta_{i}^m\,.
It can be shown that $\tau_i\in \bfD_m$, $i=1,...,n$ satisfy
the standard quadratic relation $\tau_i^2=(t^{-1}-t)\tau_i+1$, and the braid relations. So they generate the Hecke algebra $H$ over $R$. 
%The same works for $\tau^\vee_i$, which generate (another but isomorphic) Hecke algebra over $R$. 

\begin{rema}\label{rem:DL}
Let $y=-t^{-2}$.  As operators on $\bfD_m^*\cong K_T(G/B)$, then $t^{-1}\tau_i\odot\_$ agrees with $\cT_i^L$ and $t^{-1}\tau_i\bullet\_$ agrees with $\cT_i^{R, \vee}$, respectively, where the latter are notions from \cite[Section~5.3]{MNS}. 
\end{rema}
%Moreover,  $t^{-1}\tau_i$ coincides with  the corresponding Demazure-Lusztig operator in \cite[Equation (8.1)]{L85} and  \cite[Corollary 2.6]{MS19}. The operator $t^{-1}\tau_i^\vee$ appeared in \cite[Equation (3)]{BBL15} and \cite[Section 2]{LLL17} in relation to Whittaker functions. 

%Recall that we have an involution $\iota:Q_{m,W}\to Q_{m,W}$, defined above. It is easy to check that 
%\[\iota(\tau_i)=\tau^\vee_i\,.\]
%Therefore, $\iota(\tau_w)=\tau_{w^{-1}}^\vee$. Consequently, from Lemma \ref{lem:inv}, we have
%\begin{equation}\label{eq:adjtau}
%Y_\Pi^m\bullet ((\tau_i\bullet f)\cdot g)=Y_\Pi^m\bullet(f\cdot (\tau^\vee_i\bullet g)), \quad f,g\in Q_{m,W}^*\,. 
%\end{equation}

\subsubsection*{The Kazhdan-Lusztig basis} Consider the  involution of the Hecke algebra $H\to H$, $z\mapsto \overline{z}$ such that 
\begin{equation}\label{eq:inv2}
\overline{t}=t^{-1}, \quad \overline{\tau_i}=\tau_i^{-1}\,.
\end{equation}
There is a basis of $H$ over $R$ denoted by $\{\gamma_w\}_{w\in W}$ and called the Kazhdan-Lusztig basis. It is invariant under this involution and satisfies
\[
\gamma_w\in \tau_w+\sum_{v<w}t\bbZ[t]\tau_v\,.
\]
We set $t_w=t^{\ell(w)}$ and 
\[
\gamma_w=\sum_{v\le w}t_wt_v^{-1}P_{v,w}(t^{-2})\tau_v\,,
\]
where   $P_{v,w}$ are the Kazhdan-Lusztig polynomials. 
In addition to this, there is another canonical basis defined by (see~\cite{KL79})
\[\wt\gamma_w:=\sum_{v\in W}\ep_w\ep_vt_w^{-1}t_vP_{v,w}(t^2)\tau_v\in \tau_w+\sum_{v<w}t^{-1}\bbZ[t^{-1}]\tau_v\,.\]

More generally, for $J'\subset J\subseteq\Pi$, denote
\begin{equation}\label{defgjjp}
\gamma_J:=\gamma_{w_J}=\sum_{v\le w_J}t_{w_J}t_v^{-1}\tau_v\,, \quad \gamma_{J/J'}:=\sum_{v\in W_J\cap W^{J'}}t_{w_{J/J'}}t_v^{-1}\tau_v\,. 
\end{equation}
It is not difficult to see that 
\begin{equation}\label{mult-gamma}\gamma_J=\gamma_{J/J'}\gamma_{J'}\,.
\end{equation}
If $Q\subset P$ are the parabolic subgroups corresponding to $J'\subset J$, respectively, denote $\gamma_{P/Q}=\gamma_{J/J'}$.  It will be used in considering KL-Schubert classes in hyperbolic cohomology of partial flag varieties below.

\subsubsection*{Motivic Chern classes} 
We recall the definition of the motivic Chern classes, following \cite{BSY10, FRW18, AMSS19}. Let $X$ be a quasi-projective, non-singular, complex algebraic variety with an action of the torus $T$. Let ${G}_0^T(var/X)$ be the (relative) Grothendieck group  of varieties over $X$. By definition, it is the free abelian group generated by isomorphism classes $[f: Z \to X]$ where $Z$ is a quasi-projective $T$-variety and $f$ is a $T$-equivariant morphism modulo the usual additivity relations 
$$[f: Z \to X] = [f: U \to X] + [f:Z \setminus U \to X]\,,$$ for any $T$-invariant open subvariety $U \subset Z$. 
\begin{thm}\label{thm:MCdef}There exists a unique natural transformation $\MC_{-t^{-2}}: G_0^T(var/X) \to K_T(X)[t^{-2}]$ satisfying the following properties:
\begin{enumerate} \item[(1)] It is functorial with respect to $T$-equivariant proper morphisms of non-singular, quasi-projective varieties. 
\item[(2)] It satisfies the normalization condition \[ \MC_{-t^{-2}}[id_X: X \to X] = \sum (-1)^it^{-2i} [\wedge^i T^*_X] =: \lambda_{-t^{-2}}(T^*_X) \in K_T(X)[t^{-2}]\,. \]
\end{enumerate}
\end{thm}
The non-equivariant case is proved in \cite{BSY10}, and the equivariant case is shown in \cite{AMSS19,FRW18}. 

Let 
\[\mathcal{D}(-):=(-1)^{\dim X}\RHom_{\calO_{X}}(-,\omega_{X})\] 
be the Serre-Grothendieck duality functor on $K_T(X)$, where $\omega_X:=\bigwedge^{\dim X}T^*_X$ is the canonical bundle of $X$. Extend it to $K_T(X)[t^{\pm 1}]$ by setting $\mathcal{D}(t^i)=t^{-i}$. 
\begin{dfn}\label{def:smc}
Let $Z\subset X$ be a $T$-invariant subvariety.
\begin{enumerate}
\item 
Define the motivic Chern class of $Z$ to be 
\[\MC_{-t^{-2}}(Z):=\MC_{-t^{-2}}([Z\hookrightarrow X])\,.\]
\item
Further assume that  $Z$ is pure-dimensional. Define the Segre motivic Chern classes of $Z$ as follows (see \cite[Definition 6.2]{MNS}),
\[\SMC_{-t^{-2}}(Z):=t^{-2\dim Z}\frac{\mathcal{D}(\MC_{-t^{-2}}(Z))}{\lambda_{-t^{-2}}(T^*_X)}\,.\]
\end{enumerate}
\end{dfn}

\subsubsection*{Smoothness of  Schubert varieties} 
Consider the variety of complete flags $G/B$.
Let $X(w)^\circ:= BwB/B$ and $Y(w)^\circ:= B^-wB/B$ be the Schubert cells. The closures $X(w):= \overline{X(w)^\circ}$, $Y(w):= \overline{Y(w)^\circ}$ are the Schubert 
varieties. Observe that $u \le v$ with respect to the Bruhat order if and only if $X(u) \subset X(v)$. Let $\pt^m_w=w(x_\Pi) f_w^m$ denote the class of a point in $Q^*_{m,W}$ and let $\fe_w$ denote the respective $T$-fixed point in $G/B$. 

The key property of the motivic Chern classes of the Schubert cells that we need are listed below.

\begin{thm} \label{thm:smoothmc}
{\rm (1)} \cite[Theorem 7.6]{MNS}
For any $w\in W$, we have
\[\MC_{-t^{-2}}(X(w)^\circ)=t_w^{-1}\tau_w\odot \pt^m_e\,.\]

{\rm (2)} \cite[Theorem 9.1]{AMSS19}  For any $u\leq w\in W$, the Schubert variety $X(w)$ is smooth at $\fe_u$ if and only if 
\[\MC_{-t^{-2}}(X(w))|_u=\prod_{\alpha>0,\;us_\alpha\nleq w}(1-e^{u\alpha})\prod_{\alpha>0,\;us_\alpha\leq w}(1-t^{-2}e^{u\alpha})\,.\]
\end{thm}

\begin{rema}{\rm 
This theorem  is used to prove the Bump, Nakasuji and Naruse's conjectures about Casselman basis in unramified principal series representations, see \cite{BN11,BN19,N14,AMSS19,Su19}.}
\end{rema}

\begin{proof}
The first part follows from the reference mentioned. The second one follows from the fact $w_0\odot(\MC_{-t^{-2}}(Y(w)))=\MC_{-t^{-2}}(X(w_0w))$. 
%Here $w_0^L$ denotes the left action by the longest element $w_0\in W$. 
\end{proof}

Given $w\in W$, define the coefficients $a_{w,u}\in Q_m$ by the following formulas:
\begin{equation}\label{equ:coeffa}
\Gamma_w:=\sum_{v\leq w}t_v^{-1}\tau_{v}=\sum_{u\leq w} a_{w,u}\delta_{u}^m\in Q_{m,W}\,.
\end{equation}
Note that if the Schubert variety $X(w)$ is smooth, then $P_{v,w}=1$ for all $v\le w$, so $\Gamma_w=t^{-1}_w\gamma_w$. 
It is immediate to get the following corollary from Theorem~\ref{thm:smoothmc}.
\begin{cor}\label{cor:smooth}
For any $u\leq w\in W$, the Schubert variety $X(w)$ is smooth at the fixed point $\fe_u$ if and only if 
\[a_{w,u}=\prod_{\alpha>0,\;us_\alpha\leq w}\frac{1-t^{-2}e^{u\alpha}}{1-e^{u\alpha}}\,.\]
\end{cor}
\begin{proof}
By Theorem \ref{thm:smoothmc}~(1) and \eqref{equ:coeffa}, we have
\begin{align*}
\MC_{-t^{-2}}(X(w))=&\sum_{v\leq w}\MC_{-t^{-2}}(X(v)^\circ)=\sum_{v\leq w}t_v^{-1}\tau_v\odot \pt^m_e\\
=&\sum_{v\leq w} a_{w,v}\delta_v^m\odot \pt^m_e=\sum_{v\leq w} a_{w,v}\prod_{\alpha>0}(1-e^{v\alpha})f_v\,.
\end{align*}
Thus, we have
\[\MC_{-t^{-2}}(X(w))|_u=a_{w,u}\prod_{\alpha>0}(1-e^{u\alpha})\,.\]
The corollary follows from this and Theorem~\ref{thm:smoothmc}~(2). 
\end{proof}

\section{Dual bases in $K$-theory and characteristic classes of mixed Hodge modules}\label{sec:dual}
In this section, we use the two  Kazhdan-Lusztig bases of the Hecke algebra to define two collections  of classes in $K$-theory, and show that they are actually dual to each other.  We also give a geometric interpretation of one of these collections using the intersection homology mixed Hodge modules. These are also generalized to the partial flag variety case.
%Note that we do not call these classes KL-Schubert classes, since it is unclear to us how they are related to the $K$-theory Schubert classes.

\subsubsection*{$K$-theory KL classes}

\begin{dfn}\label{def:KL2}
We define two collections of classes (called KL classes) in $K_T(G/B)[t^{\pm 1}]$ as follows:
\[\kl_w:=\gamma_w\odot \pt_e^m\,,\quad\;\;\klt_w:=\wt\gamma_{w^{-1}w_0}\bullet\pt_{w_0}^m \,.\] 
\end{dfn}
They form a basis of the localized $K$-theory  $Q_m\otimes_{S_m}K_T(G/B)$. 

Let $\langle-,-\rangle$ denote the usual non-degenerate tensor product pairing on $K_T(G/B)[t^{\pm 1}]$, i.e., $\langle f, g\rangle=Y_\Pi^m\bullet (f\cdot g), f, g\in K_T(G/B)[t^{\pm 1}]$. The first main result of this section is the following.
\begin{thm}\label{thm:dual}
For any $w,v\in W$, we have
\[\langle \kl_w, \klt_v\rangle=\delta_{w,v}^{Kr}\prod_{\alpha>0}(t-t^{-1}e^{-\alpha})\,.\]
\end{thm}

We first recall that the Segre motivic Chern classes of  Schubert cells enjoy the following properties. 
\begin{lem}\label{lem:smc}
{\rm (1)} 
For any $v\in W$, we have
\[(\tau_{w_0v})^{-1}\bullet \pt_{w_0}^m=t_{w_0v}\prod_{\alpha>0}(1-t^{-2}e^{-\alpha}) \SMC_{-t^{-2}}(Y(v)^\circ)\,.\]

{\rm (2)} 
For any $u,v\in W$, we have
\[\left\langle \MC_{-t^{-2}}(X(u)^\circ),\SMC_{-t^{-2}}(Y(v)^\circ)\right\rangle=\delta_{u,v}^{Kr}\,.\]
\end{lem}
\begin{proof}
The first part follows from Remark \ref{rem:DL} and \cite[Theorem 7.4]{MNS}, while the second one follows from Theorem~7.1 of \textit{loc. cit.}.
\end{proof}

\begin{rema}{\rm 
By definition, $(t^{-1}\tau_i)|_{t=\infty}=Y_i^m-1$. Thus, from Theorem \ref{thm:smoothmc}(1), we get
\begin{align*}
\MC_{-t^{-2}}(X(w)^\circ)|_{t=\infty}=&t_w^{-a}\tau_w\odot \pt_e^m|_{t=\infty}
=[\mathcal{O}_{X(w)}(-\partial X(w))]=:\mathcal{I}_w\,,
\end{align*}
where $\partial X(w)=\cup_{v<w}X(v)$ is the boundary of the Schubert variety $X(w)$, and $\mathcal{I}_w$ denotes its ideal sheaf. On the other hand, $(t^{-1}\tau_i^{-1})|_{t=\infty}=Y_i^m.$ Thus, the first part of the lemma gives
\[\SMC_{-t^{-2}}(Y(v)^\circ)|_{t=\infty}=(t_{w_0v}\tau_{w_0v})^{-1}\bullet \pt_{w_0}^m|_{t=\infty}=[\mathcal{O}_{Y(v)}]\,.\]
Therefore, setting $t=\infty$ in the second part of the lemma, we get the classical fact
\[\left\langle \mathcal{I}_w, [\mathcal{O}_{Y(v)}]\right\rangle=\delta_{u,v}^{Kr}\,.\]}
\end{rema}

\subsubsection*{Proof of Theorem {\rm \ref{thm:dual}}}
First of all, we have the following inversion formula for the Kazhdan-Lusztig polynomials (see \cite[Theorem~3.1]{KL79}):
\begin{equation*}
\sum_{z}\epsilon_y\epsilon_z P_{x,z}P_{w_0y,w_0z}=\delta_{x,y}^{Kr}\,.
\end{equation*}
Therefore,
\begin{equation}\label{equ:inversion}
\sum_{z}\epsilon_x\epsilon_z P_{w_0z,w_0x}P_{z,y}=\delta_{x,y}^{Kr}\,.
\end{equation}

By definition and Theorem~\ref{thm:smoothmc}~(1), 
\begin{equation}\label{equ:cw}
\kl_w=\sum_{u\leq w}t_wt_u^{-1}P_{u,w}(t^{-2})\tau_u\odot \pt_e^m=\sum_{u\leq w}t_wP_{u,w}(t^{-2})\MC_{-t^{-2}}(X(u)^\circ)\,.
\end{equation} 
On the other hand, since $\wt\gamma_w$ is invariant under the involution, we get
\[\wt\gamma_w=\sum_{v\in W}\ep_w\ep_vt_wt_v^{-1}P_{v,w}(t^{-2})\tau_{v^{-1}}^{-1}\,.\]
Thus, 
\begin{align}\label{equ:cwt}
\klt_w=&\wt\gamma_{w^{-1}w_0}\bullet\pt_{w_0}^m\nonumber\\
=&\sum_{v\geq w}\ep_w\ep_vt_{w^{-1}w_0}t_{v^{-1}w_0}^{-1}P_{v^{-1}w_0,w^{-1}w_0}(t^{-2})\tau_{w_0v}^{-1}\bullet\pt_{w_0}^m\nonumber\\
=&\prod_{\alpha>0}(1-t^{-2}e^{-\alpha})\sum_{v\geq w}\ep_w\ep_vt_{w^{-1}w_0}P_{v^{-1}w_0,w^{-1}w_0}(t^{-2})\SMC_{-t^{-2}}(Y(v)^\circ)\,,
\end{align}
where the last step follows from Lemma \ref{lem:smc}~(1).

Therefore, we have
\begin{align*}
\langle \kl_w, \klt_y\rangle =&\prod_{\alpha>0}(1-t^{-2}e^{-\alpha})t_wt_{y^{-1}w_0}\sum_{u}P_{u,w}\sum_{v}\epsilon_v\epsilon_y  P_{v^{-1}w_0,y^{-1}w_0}\delta_{u,v}^{Kr} \\
=&\prod_{\alpha>0}(1-t^{-2}e^{-\alpha})t_wt_{y^{-1}w_0}\sum_{u}P_{u,w}\epsilon_u\epsilon_y  P_{w_0u,w_0y} \\
=&\prod_{\alpha>0}(t-t^{-1}e^{-\alpha})\delta_{w,y}^{Kr}\,,
\end{align*}
where the first equality follows from Lemma \ref{lem:smc}~(2), the second follows from $P_{u,v}=P_{u^{-1}, v^{-1}}$, and the third one follows from \eqref{equ:inversion}.

An immediate corollary of the proof is the following.
\begin{cor}
If the Schubert variety $X(w)$ is smooth, 
\[\kl_w=\sum_{u\leq w}t_w\MC_{-t^{-2}}(X(u)^\circ)=t_w\MC_{-t^{-2}}(X(w))\in K_T(G/B)[t^{\pm 1}]\,.\]
\end{cor}
\begin{proof}
It follows directly from \eqref{equ:cw} and the fact $P_{u,w}=1$ for all $u\le w$. 
\end{proof}

\subsubsection*{Characteristic classes of mixed Hodge modules}
For any parabolic subgroup $P_J$, let $K^0(\MHM(G/P_J,B))$ denote its Grothendieck group of $B$-equivariant mixed Hodge modules. Recall there is a motivic Hodge Chern transformation (see \cite[Definition 5.3 and Remark 5.5]{S11})
\[\MHC_{-t^{-2}}:K^0(\MHM(G/P_J,B))\rightarrow K_B(G/P_J)[t^{\pm 1}]\simeq K_T(G/P_J)[t^{\pm 1}]\,,\]
such that for any $[f:Z\rightarrow G/P_J]\in G_0^B(var/({G/P_J}))$,
\begin{equation}\label{equ:MHC}
\MC_{-t^{-2}}([f:Z\rightarrow G/P_J])=\MHC_{-t^{-2}}([f_!\bbQ_Z^H])\,,
\end{equation}
where $[\bbQ_Z^H]:=[k^*\bbQ_{\pt}^H]\in K^0(\MHM(Z,B))$ and $k:Z\rightarrow \pt$ is the structure morphism. The construction also works for $B^-$-equivariant mixed Hodge modules, where $B^-$ is the opposite Borel subgroup. The natural transformation $\MC_{-t^{-2}}$ commutes with the Serre-Grothendieck dual as follows, see Corollary 5.19 of \textit{loc. cit.},
\begin{equation}\label{equ:D}
\MHC_{-t^{-2}}\circ \mathcal{D}=\mathcal{D}\circ \MHC_{-t^{-2}}\,,
\end{equation}
where the $\mathcal{D}$ on the left hand side is the dual of the mixed Hodge modules, while the other one is the Serre-Grothendieck dual. Here both are denoted by $\calD$, if no confusion is possible.

For any $u\in W$, let $i_u:X(u)^\circ\hookrightarrow G/B$ and $j_u:Y(u)^\circ\hookrightarrow G/B$ be the inclusions. Then by \eqref{equ:MHC}
\[\MC_{-t^{-2}}(X(u)^\circ)=\MHC_{-t^{-2}}([i_{u!}\bbQ_{X(u)^\circ}^H])\,.\]
Since $\calD\circ j_{v!}=j_{v*}\circ\calD$, and
\[\calD(\bbQ_{Y(v)^\circ}^H)=\bbQ_{Y(v)^\circ}^H[2\dim Y(v)^\circ](\dim Y(v)^\circ)\,,\]
where $[2\dim Y(v)^\circ]$ means shift by $2\dim Y(v)^\circ$ and $(\dim Y(v)^\circ)$ denotes the twist by the Tate Hodge module $\bbQ^H(1)^{\otimes \dim Y(v)^\circ}$,
Equation \eqref{equ:D} gives
\[\SMC_{-t^{-2}}(Y(v)^\circ)=\frac{\MHC_{-t^{-2}}([j_{v*}\bbQ_{Y(v)^\circ}^H])}{\lambda_{-t^{-2}}(T^*_{G/B})}\,.\]
Using these, Lemma \ref{lem:smc}(2) can also be proved using mixed Hodge modules by J. Sch\"urmann. For the analogue in equivariant homology, see \cite[Theorem 1.2]{S17}.

For any Schubert variety $X(w)$, let $[\IC_{X(w)}^H]\in K^0(\MHM(G/B,B))$ denote the intersection homology Hodge module on $X(w)$. Then it is well known that (see \cite{KL80,T87,KT02}), 
\[[\IC_{X(w)}^H]=\sum_{u\leq w}\ep_wP_{u,w}(t^{-2})[i_{u!}\bbQ_{X(u)^\circ}^H]\,.\]
Thus, 
\[\MHC_{-t^{-2}}([\IC_{X(w)}^H])=\sum_{u\leq w}\ep_wP_{u,w}(t^{-2})\MC_{-t^{-2}}(X(u)^\circ)\,.\]
Comparing with \eqref{equ:cw}, we get the following geometric interpretation of the KL classes $C_w$ in Definition \ref{def:KL2}.
\begin{prop}\label{prop:cwIC}
For any $w\in W$,
\[\kl_w=t_w\ep_w\MHC_{-t^{-2}}([\IC_{X(w)}^H])\in K_T(G/B)[t^{\pm 1}]\,.\]
\end{prop}
An immediate Corollary is the following. 
\begin{cor}\label{cor:inv}
The canonical basis $\kl_w$ is invariant under the Serre-Grothendieck duality, i.e.,
\[\calD(\kl_w)=\kl_w\in K_T(G/B)[t^{\pm 1}]\,.\]
\end{cor}
\begin{proof}
Since 
\[\calD(\IC_{X(w)}^H)=\IC_{X(w)}^H(\dim X(w))\,,\]
Equation \eqref{equ:D} and Proposition \ref{prop:cwIC} give
\begin{align*}
\calD(\kl_w)=&\calD(t_w\ep_w\MHC_{-t^{-2}}([\IC_{X(w)}^H])=t_w^{-1}\ep_w\MHC_{-t^{-2}}(\calD([\IC_{X(w)}^H]))=\kl_w\,.
\end{align*}
\end{proof}

\subsubsection*{Parabolic case}
In this subsection, we generalize the above results to the parabolic case. Let $J\subset \Pi$ be a subset of simple roots, with corresponding parabolic subgroup $P_J$. Schubert cells and varieties and opposite Schubert cells and varieties of $G/P_J$ are indicated by subscripts $J$. 
% Let $R^+_J\subset R^+$ be the subset of positive roots spanned by the simple roots in $J$.
 Recall there exist parabolic Kazhdan-Lusztig polynomials (see \cite{D79,KT02}), denoted by $P_{v,w}^J\in \bbZ[t^{-2}]$, where $v,w\in W^J$. Here our $P_{v,w}^J$ is the $u=-1$ parabolic KL polynomials in \cite{D79}, which is also denoted by $P_{v,w}^{J,q}$ in \cite[Remark 2.1]{KT02}. We have the following property, which generalizes \cite[Proposition 3.4]{D79}.
\begin{lem} \cite[Proposition 5.19]{LZZ}\label{lemma:pkl}
For any $w,v\in W^J$ and $u\in W_J$,
\[P_{vu,ww_J}=P_{v,w}^J\,.\]
\end{lem}

Let $Q_{u,w}:=P_{w_0w,w_0u}$ denote the usual inverse KL polynomials, which satisfy 
\[\sum_w \ep_u\ep_w Q_{u,w}P_{w,v}=\delta_{u,v}^{Kr}\,.\] 
For any $u,w\in W^J$, let $Q_{u,w}^J\in \bbZ[t^{-2}]$ denote the inverse parabolic KL polynomial (see \cite{KT02}\footnote{Our $Q_{u,w}^J$ is denoted by $Q_{u,w}^{J,q}$ in \cite{KT02}.}). Then 
\begin{equation}\label{equ:inversekl}
\sum_{w\in W^J} \ep_u\ep_w Q^J_{u,w}P^J_{w,v}=\delta_{u,v}^{Kr}\,.
\end{equation}
Moreover, it is related to the usual $Q_{u,w}$ as follows, see \cite[Proposition 2.6]{KT02} or \cite{S97}:
\[Q_{u,w}^J=\sum_{v\in W_J}\epsilon_v\epsilon_{w_J}Q_{uw_J,wv}\,.\]

Following Equations \eqref{equ:cw} and \eqref{equ:cwt}, we define the parabolic canonical bases in $K_T(G/P_J)[t^{\pm 1}]$ as follows.
\begin{dfn}
For any $w\in W^J$, let 
\[\kl^J_w:=\sum_{u\in W^J, u\leq w}t_wP^J_{u,w}(t^{-2})\MC_{-t^{-2}}(X(u)_J^\circ)\,,\]
and 
\[\klt^J_w:=\prod_{\alpha\in \Sigma^+\setminus \Sigma^+_J}(1-t^{-2}e^{-\alpha})\sum_{v\in W^J,v\geq w}\ep_w\ep_vt_{w_Jw^{-1}w_0}Q_{w,v}^J(t^{-2})\SMC_{-t^{-2}}(Y(v)_J^\circ)\,.
\]
\end{dfn}
Then if  $J=\emptyset$, then $C_w^\emptyset=C_w$, and $\wt C_w^\emptyset=\wt C_w$, as  defined before. 

Let $\langle -,-\rangle_J$ denote the non-degenerate tensor product pairing on $K_T(G/P_J)$. 
The parabolic analog of Lemma \ref{lem:smc}(2) also holds (see \cite[Theorem 7.2]{MNS}):
for any $u,v\in W^J$, 
\[\left\langle \MC_{-t^{-2}}(X(u)_J^\circ),\SMC_{-t^{-2}}(Y(v)_J^\circ)\right\rangle_J=\delta_{u,v}^{Kr}\,.\]
Combining this with \eqref{equ:inversekl}, we immediately get the following generalization of Theorem \ref{thm:dual}.
\begin{thm}\label{thm:KLdualP}
For any $u,w\in W^J$,
\[\langle \kl^J_w,\klt^J_u\rangle_J=\delta_{u,w}^{Kr}\prod_{\alpha\in \Sigma^+\setminus \Sigma^+_J}(t-t^{-1}e^{-\alpha})\]
\end{thm}

We now investigate the relation between KL classes of $G/B$ and $G/P_J$. 
For any $w\in W^J$, let us still use $i_u$ denote the inclusion $X(u)^\circ_J\hookrightarrow G/P_J$. Then the following identity holds in $K^0(\MHM(G/P_J, B))$ (see \cite[Corollary 5.1]{KT02}),
\[[\IC_{X(w)_J}^H]=\sum_{u\in W^J, u\leq w}\ep_w P_{u,w}^J[i_{u!}\bbQ_{X(u)^\circ_J}^H]\,.\]
Thus, we get the following parabolic analog of Proposition \ref{prop:cwIC} and Corollary \ref{cor:inv}.
\begin{prop}
For any $w\in W^J$,
\[\kl^J_w=t_w\ep_w \MHC_{-t^{-2}}([\IC_{X(w)_J}^H])\,.\]
Moreover, let $\calD_J$ denote the Serre-Grothendieck duality functor on $G/P_J$. Then
\[\calD_J(\kl^J_w)=\kl^J_w.\] 
\end{prop}

Recall $\pi_J:G/B\rightarrow G/P_J$ denotes the natural projection. The relation between $\kl_w$ and $\kl_w^J$ is given by the following proposition.
\begin{prop}
Let $\mathcal{P}_J(t)=\sum_{v\in W_J}t_v$ be the Poincar\'e  polynomial of $W_J$. then for any $w\in W^J$,
\[\pi_{J*}(\kl_{ww_J})=t_{w_J}^{-1}\mathcal{P}_J(t^2) \kl_w^J\in K_T(G/P_J)[t^{\pm 1}]\,.\]
\end{prop}
\begin{proof}
By \cite[Remark 5.5]{AMSS19}, for any $u\in W^J$ and $v\in W_J$,
\[\pi_{J*}(\MC_{-t^{-2}})(X(uv)^\circ)=t_v^{-2}\MC_{-t^{-2}}(X(u)_J^\circ)\,,\]
which also follows directly from the following identity about mixed Hodge modules 
\[\pi_{J!}(i_{uv!}\bbQ^H_{X(uv)^\circ})=\bbQ_{X(u)_J^\circ}^H[-2\ell(v)](-\ell(v))\,.\]
Thus,
\begin{align*}
\pi_{J*}(\kl_{ww_J})=&\sum_{u\in W^J, u\leq w}\sum_{v\in W_J}t_wt_{w_J}P_{uv,ww_J}\pi_{J*}\MC_{-t^{-2}}(X(uv)^\circ)\\
=&\sum_{u\in W^J, u\leq w}t_wt_{w_J}P^J_{u,w}\MC_{-t^{-2}}(X(u)_J^\circ)\sum_{v\in W_J}t_v^{-2}\\
=&\kl_w^J\sum_{v\in W_J}t_v^{-2}t_{w_J}=C_w^J\sum_{v\in W_J}t_{w_J}^{-1}t^2_{w_J}t_v^{-2}=C_w^J\sum_{v\in W_J}t_{w_J}^{-1}t_{vw_J}^{2}=C_w^Jt_{w_J}^{-1}\mathcal{P}_J(t^2)\,, 
\end{align*}
where the second equality follows from Lemma \ref{lemma:pkl}.
\end{proof}

\section{The smoothness conjecture for hyperbolic cohomology}
In this section, we use the smoothness criterion to prove the Smoothness Conjecture. Since we will be working with multiplicative and hyperbolic formal group laws in the same time, we add superscripts or subscripts $m$ (resp. $t$) in the multiplicative case (resp.  hyperbolic case). 
\subsubsection*{The hyperbolic case} Set $\mu=t+t^{-1}$.  
Consider the hyperbolic formal group law over $R$
\[
F_t(x,y):=\tfrac{x+y-xy}{1-\mu^{-2}xy}\,.
\] The definitions of Section~\ref{sec:DEM} applied to $F_t$
give the respective rings
\[
S_t,\; Q_t,\; Q_{t,W},\; \bfD_t\,. 
\]

Consider a map of formal group laws
\[
g\colon F_t\to F_m\,, \quad g(x)=\tfrac{(1-t^2)x}{x-(t^2+1)}\,, 
\]
so that $F_m(g(x), g(y))=g(F_t(x,y))$. It induces ring embeddings
\[\psi\colon S_m\hookrightarrow S_t\,, \quad \psi(f(x_\la))=f(g(x_\la)), \quad f(x)\in R[[x]]\,, 
\] and  
\begin{equation}\label{eq:psi}
\psi\colon Q_m\hookrightarrow R\left[\tfrac{1}{1-t^2}\right]\otimes Q_t\,.
\end{equation}
Consequently, we have a ring embedding
\[
\psi:Q_{m,W}\to R\left[\tfrac{1}{1-t^{2}}\right]\otimes_R Q_{t, W}\,,  \quad \psi(p\de_w^m)=\psi(p)\de_w^t\,, \quad p\in Q_m\,, w\in W\,.
\]
It can be shown that 
\begin{equation}\label{eq:psi2}
\psi(\tau_i)=\mu Y_i^t-t\in \bfD_t \subset Q_{t,W}\,.
\end{equation}
Note that in \eqref{eq:psi},  for the target, we have to invert $t^2-1$, but for the one in \eqref{eq:psi2}, it is not necessary. 

One of the most interesting properties of $\psi$ is the following  (see \cite[Corollary~5.5~(2)]{LZZ}): 
\begin{equation}  \label{eq:gammapsirel}
\mu^{-\ell(w_{J/J'})}\psi(\gamma_{J/J'})Y^t_{J'}=Y^t_J\,.
\end{equation}
In other words, $\psi(\gamma_{J/J'})$ behaves like a replacement of $Y_{J/J'}$; see \cite[Remark 5.6]{LZZ}.  In particular, letting $J'=\emptyset$, one then has
\[\mu^{-\ell(w_J)}\psi(\gamma_{w_J})=Y_J^t\,.\]

Let $\fh$ denote the respective oriented cohomology theory for the hyperbolic formal group law $F_t$. 

\begin{dfn}\label{def:KL} Define the KL-Schubert class for $w\in W^J$ to be 
\[
\KL_w^J:=\mu^{-\ell(ww_J)}\psi(\gamma_{ww_J})\odot\pt^t_e\in (\bfD_{t}^*)^{W_J}\cong \fh_T(G/P_J)\,.
\]
\end{dfn}

%One can also define such KL-Schubert  classes in the multiplicative case (in $K_T(G/P_J)$ over $R$) by dropping $\psi$, 
%i.e. by setting $\KL_w^J:=\mu^{-\ell(ww_J)} \gamma_{ww_J}\odot\pt^m_e$. 

\begin{rema}{\rm 
Following \cite{LZZ} one can define certain involution on some subset $\calN_J:=\psi(H)\odot \pt^t_e\subset\bfD_t^*$ so that $\KL_v^J$ is invariant under such involution, similar to the parabolic Kazhdan-Lusztig basis of Deodhar. }
\end{rema}

We now prove the Smoothness Conjecture \cite[Conjecture 5.14]{LZZ}. Several special cases were proved in~\cite{LZ17,LZZ}, such as the case of $w=w_{J/J'}$ for $J'\subset J\subseteq\Pi$ (i.e., $w$ has `relative' maximal length), and that of Schubert varieties in complex projective spaces.

\begin{thm}\label{thm:main} If $X(w)$ is smooth, then the class determined by $X(w)$ in $\fh_T(G/B)$ coincides with the KL-Schubert class $\KL_w$.
\end{thm}
\begin{proof}
Since $X(w)$ is smooth, $P_{v,w}=1$ for any $v\le w$, see \cite[6.1.19]{BL00}. Therefore, 
\begin{align*}
\gamma_{w}=\sum_{v\le w}t_wt_v^{-1}\tau_v=t_w\sum_{v\le w}t_v^{-1}\tau_v
=t_w\Gamma_w=t_w\sum_{v\le w}a_{w,v}\de_v^m\,.
\end{align*}
From the definition of $\psi$, it is easy to verify that 
\begin{equation}\label{eq:psi1}
\psi\left(\frac{1-t^{-2}e^\al}{1-e^\al}\right)=\frac{t^{-1}\mu}{x_{-\al}}\,. 
\end{equation}
Then for any $w\in W$, we have
\begin{align*}
\KL_w&=\mu^{-\ell(w)}\psi(\gamma_{w})\odot \pt^t_e\\
&=\mu^{-\ell(w)}\psi\left(t_w\sum_{v\le w}a_{w,v}\de_v^m\right)\odot \pt_e^t\\
&\overset{\text{Cor. }\ref{cor:smooth}}=\mu^{-\ell(w)}t_w\sum_{v\le w}\psi\left(\prod_{\al>0,\; vs_\al\le w}\frac{1-t^{-2}e^{u\al}}{1-e^{u\al}}\right)\de_v^t\odot \pt_e^t\\
&\overset{\eqref{eq:act},\eqref{eq:psi1}}=\mu^{-\ell(w)}t_w\sum_{v\le w}\left(\prod_{\al>0,\; vs_\al\le w}\frac{t^{-1}\mu}{x_{-v\al}}\right)\cdot v(x_\Pi^t)f_v^t\\
&=\sum_{v\le w}v\left(\frac{\prod_{\al<0}x_\al}{\prod_{\al<0,\; vs_\al\le w}x_{\al}}\right)f_v^t\\
&=\sum_{v\le w}\frac{\prod_{\al>0}{x_{-\al}}}{\prod_{\al>0,\; s_\al v\le w}x_{-\al}}f_v^t\,. 
\end{align*}
Here the fifth identity follows from the following well-known fact:
\[
\textit{for any } v\leq w\in W, \textit{ if  X(w) is smooth,\quad} 
|\{\al>0\mid s_\al v\le w\}|=\ell(w)\,,
\]
and the last one is proved as follows: for any $v\leq w\in W$,
\begin{align*}
\frac{\prod_{\al<0}x_{v\al}}{\prod_{\al<0,\; vs_\al\le w}x_{v\al}}=&\frac{\prod_{\alpha>0,\;s_\alpha v<v}x_\alpha \cdot \prod_{\alpha>0,\;v<s_\alpha v}x_{-\alpha}}{\prod_{\alpha>0,\;s_\alpha v<v}x_\alpha \cdot \prod_{\alpha>0,\;v<s_\alpha v\leq w}x_{-\alpha}}\\
=&\frac{\prod_{\alpha>0,\;s_\alpha v<v}x_{-\alpha}\cdot \prod_{\alpha>0,\;v<s_\alpha v}x_{-\alpha}}{\prod_{\alpha>0,\;s_\alpha v<v}x_{-\alpha}\cdot \prod_{\alpha>0,\;v<s_\alpha v\leq w}x_{-\alpha}}\\
=&\frac{\prod_{\alpha>0}x_{-\alpha}}{\prod_{\alpha>0,\;s_\alpha v\leq w}x_{-\alpha}}\,.
\end{align*}
Comparing with the restriction formula of $[X(w)]$ in \cite[(5.6)]{LZZ}, we  see that $\KL_w=[X(w)]$. The proof is finished. 
\end{proof}

%\begin{rema}{\rm 
%In \cite[(24)]{LZ17}, the restriction formula of $[X(w)]$ was written in a different form. For the convenience of the reader, we show here that they are equivalent. 
%\begin{align*}
%\frac{\prod_{\be<0}x_{v\be}}{\prod_{\al<0,\; vs_\al\le w}x_{v\al}}&=\frac{\prod_{\be>0}x_{-v\be}}{\prod_{\al>0,\; vs_\al\le w}x_{-v\al}}\\
%&=\frac{\prod_{\be>0}x_{-v\be}}{\prod_{\al>0,\; v<vs_\al\le w}x_{-v\al}\prod_{\al>0,\; vs_\al<v}x_{-v\al}}\\
%&=\frac{\prod_{\be>0,\;v\beta>0}x_{-v\be}}{\prod_{\al>0,\; v<s_\al v\le w}x_{-\al}}\\
%&=\frac{\prod_{\be>0,\;v\beta>0}x_{-v\be}\prod_{\be>0,\;v\beta<0}x_{v\be}}{\prod_{\al>0,\; v<s_\al v\le w}x_{-\al}\prod_{\be>0,\;v\beta<0}x_{v\be}}\\
%&=\frac{\prod_{\be>0}x_{-\be}}{\prod_{\al>0,\; v<s_\al v\le w}x_{-\al}\prod_{\al>0,\; s_\al v<v}x_{-\al}}\\
%&=\frac{\prod_{\alpha>0}x_{-\alpha}}{\prod_{\alpha>0,\;s_\alpha v\leq w}x_{-\alpha}}\,.
%\end{align*}}
%\end{rema}

We now look at the case of partial flag varieties. Let $P_J$ be the parabolic subgroup with the projection map $\pi_J:G/B\to G/P_J$. Let $w_J$ be the longest element in the subgroup $W_J$ of $W$ determined by $J$, and $W^J\subset W$ be the set of minimal length representatives of $W/W_J$. Recall $X(w)_J$ denotes the Schubert variety of $G/P_J$ determined by $w\in W^J$.

For $G/P_J$, the definition of KL-Schubert class $\KL_w^J$ corresponding to $w\in W^J$ is defined by using the so-called parabolic Kazhdan-Lusztig basis. According to the paragraph right after \cite[Definition 5.9]{LZZ}, via the embedding $\pi_J^*:\fh_T(G/P_J)\to \fh_T(G/B)$,  we have 
\[
\pi_J^*(\KL_w^J)=\KL_{ww_J}\,. 
\]
\begin{cor}\label{cor:mainP}
Conjecture $5.14$ of \cite{LZZ} holds for any partial flag variety $G/P_J$, that is, if the Schubert variety $X(w)_J$ of $G/P_J$ is smooth for $w\in W^J$, then the KL-Schubert class $\KL_w^J$ of $w$ coincides with the fundamental class $[X(w)_J]$. 
\end{cor}
\begin{proof}
We have the following Cartesian diagram:
\[
\xymatrix{\pi_J^{-1}(X(w)_J)\ar[r]^{i'} \ar[d]^{\pi_J} & G/B\ar[d]^{\pi_J}\\
X(w)_J \ar[r]^i & G/P_J\,.}
\]
Moreover, $\pi_J^{-1}(X(w)_J)=X(ww_J)$. Since $X(w)_J$ is smooth, $X(ww_J)$ is also smooth. Thus, Theorem \ref{thm:main} implies $[X(ww_J)]=\KL_{ww_J}$. On the other hand, we get the following by proper base change:
\[
\pi_J^*[X(w)_J]=\pi_J^*i_*[1_{X(w)_J}]=i'_*\pi_J^*[1_{X(w)_J}]=i_*'[1_{X(ww_J)}]=[X(ww_J)]\,,
\]
where the third equality follows from the fact that the pull-back $\pi_J^*$ preserves identity. Since $\pi_J^*(\KL_w^J)=\KL_{ww_J}$ and $\pi_J^*$ is injective, we get $\KL_w^J=[X(ww_J)]\in \fh_T(G/P_J)$.
\end{proof}

\section{KL-Schubert classes and small resolutions}

In this section, we give a geometric interpretation of the KL-Schubert classes (for hyperbolic cohomology) in the case of type $A$ Grassmannians.

For subsets $J'\subset J\subseteq \Pi$, for hyperbolic cohomology,  we will use relative push-pull elements $Y_{J/J'}^t$ defined in \eqref{eq:relpush}. For simplicity, we will skip the superscript $t$. Moreover, if $Q\subset P$ are the parabolic subgroups corresponding to $J'\subset J$, respectively, we will denote $Y_{P/Q}=Y_{J/J'}$. 

Consider the Grassmannian $Gr_d(\bC^{n-d})=SL_n/P_J$, where the set of simple roots $\Pi$ is identified with $\{1,\ldots,n-1\}$ and $J:=\Pi\setminus\{d\}$. Fix a Schubert variety $X(\lambda)$ of it, which is  indexed by a partition $\lambda=(\lambda_1\ge\ldots\ge\lambda_l>0)$ contained inside the $d\times(n-d)$ rectangle; here we mean that $\lambda$ is identified with a Young diagram (in English notation), whose top left box is placed on the top left box of the mentioned rectangle. 

Alternatively, the Schubert variety $X(\lambda)$ is indexed by a $d$-subset $I_\lambda$ of $[n]:=\{1,\ldots,n\}$, which is constructed as follows. Place the above $d\times(n-d)$ rectangle inside the first quadrant of the $xy$-plane, such that its southwest corner is the origin. Label each horizontal (resp. vertical) unit segment whose left (resp. bottom) endpoint is a lattice point $(x,y)$ by $x+y+1$. Consider the lattice path from $(0,0)$ to $(n-d,d)$ defining the southeast boundary of the Young diagram $\lambda$ when embedded into the $d\times(n-d)$ rectangle as stated above. Then $I_\lambda$ consists of the labels on the vertical steps of this path.

Yet another indexing of the Schubert variety $X(\lambda)$ is by a {\em Grassmannian permutation} $w_\lambda$ in the symmetric group $W=S_n$, which has its unique descent in position $d$. Written in one-line notation, $w_\lambda$ consists of the entries in $I_\lambda$ followed by the entries in $[n]\setminus I_\lambda$, where both sets of entries are ordered increasingly. Thus, $w_\lambda$ belongs to the set $W^J$ of lowest coset representatives modulo the parabolic subgroup $W_J$. Moreover, it has the following reduced decomposition:
\begin{equation}\label{fact}w_\lambda=\prod_{(i,j)\in\lambda}^{\rightarrow} s_{d+j-i}\,;\end{equation}
here $(i,j)$ is the box of the Young diagram $\lambda$ in row $i$ and column $j$, while in the product we scan the rows of $\lambda$ from bottom to top, and each row from right to left. 

\begin{exa}\label{exlam}{\rm We use as a running example the same one as in \cite[Example~9.1.11]{BL00}, namely $n=10$, $d=5$, $\lambda=(5,5,3,2,2)$, $I_\lambda=\{3,4,6,9,10\}$. In order to illustrate~\eqref{fact}, we place the number $d+j-i$ in the box $(i,j)$ of $\lambda$, as follows:
\begin{equation}\label{contents}\tableau{5&6&7&8&9\\4&5&6&7&8\\3&4&5\\2&3\\1&2}\,.\end{equation}
Thus, we have
\begin{equation}\label{exfact}w_\lambda=[3,4,6,9,10,1,2,5,7,8]=(s_2s_1)(s_3s_2)(s_5s_4s_3)(s_8s_7s_6s_5s_4)(s_9s_8s_7s_6s_5)\,.\end{equation}
}\end{exa}

In \cite[Section~9.1]{BL00}, the permutation $w_\lambda$ is identified with the $d$-subset $I_\lambda$, and they are encoded into a $2\times m$ matrix
\begin{equation}\label{2row}\left(\begin{array}{ccc} k_1&\ldots&k_m\\a_1&\ldots&a_m\end{array}\right)\,,\end{equation}
which can be read off from the above lattice path as follows. The entries $0<k_1<\ldots<k_m\le n$ are the labels of the last steps in consecutive sequences of vertical (unit) steps. The entries $a_1,\ldots,a_m$ are the lengths of these sequences. The numbers $b_0,\ldots,b_{m-1}$ calculated in~\cite{BL00} are the lengths of the sequences of horizontal steps, where we set $b_0:=0$ if $l<d$ (i.e., if the lattice path starts with a vertical step). Recall that we also set $a_0=b_m:=\infty$. 

Now recall that the Schubert variety $X(\lambda)$ has {\em small resolutions}, which were defined by Zelevinsky~\cite{Z83}. We briefly recall their construction following~\cite[Section~9.1]{BL00}. This construction starts with the choice of an index~$i$, with $0\le i<m$, such that $b_i\le a_i$ and $a_{i+1}\le b_{i+1}$ (any such choice can be made). While it is clear that such an index always exists, we avoid the choice of $i=0$ if $l<d$. Then, a new permutation $w^2$ is obtained from $w^1:=w_\lambda$ via a certain procedure, which can be rephrased as follows. Consider the $i$-th outer corner of $\lambda$ (counting from $0$), from southwest to northeast, where the origin is an outer corner if and only if $l<d$. Consider the rectangle $R_1$ (inside $\lambda$) whose southeast vertex is the mentioned outer corner, and which is maximal such that its removal from $\lambda$ still leaves a Young diagram. It is clear that the size of $R_1$ is $b_{i}\times a_{i+1}$. Then $w^2$ is the Grassmannian permutation corresponding to the Young diagram $\lambda\setminus R_1$.

The above procedure is then iterated. We thus tile the Young diagram $\lambda$ with rectangles $R_1,\ldots,R_r$. Let us denote by $p_i$ and $q_i$ the height and width of $R_i$, respectively. We also define the sequence of Grassmannian permutations $w^1,\ldots,w^r$, such that the Young diagram of $w^i$ is $\lambda^i:=\lambda\setminus\rho^{i-1}$, where $\rho^j:=R_1\cup\ldots\cup R_j$. In particular, the Young diagram of $w^r$ is $R_r$, and the Schubert variety $X({w^r})$ is smooth. Note that $r=m$ if $l=d$, and $r=m-1$ if $l<d$.

\begin{exa}\label{exzel}{\rm We continue Example~\ref{exlam}. The encoding of $w_\lambda$ by the $2\times m$ matrix~\eqref{2row} and the successive choices of $w^1,\,w^2,\,w^3$ based on it are described in detail in~\cite{BL00}. In our setup, the tiling of $\lambda$ with the corresponding rectangles $R_1,\,R_2,\,R_3$ is illustrated below (the number in a box is the index of the rectangle to which that box belongs).
\[\tableau{3&3&2&2&2\\3&3&2&2&2\\3&3&1\\3&3\\3&3}\]
}
\end{exa}

In order to complete the construction of the Zelevinsky resolution, following~\cite[Section~9.1]{BL00}, we need the stabilizer $P_{w_\lambda}$ of the Schubert variety $X(\lambda)=X(w_\lambda)$. This is the parabolic subgroup corresponding to the subset $\Pi\setminus\{k_1,\ldots,k_m\}$, cf.~\eqref{2row}. More generally, consider the stabilizers $P_i:=P_{w^i}$, for $i=1,\ldots,r$, and  $P_{r+1}:=P_J$; for simplicity, we use the same notation for the corresponding subsets of $\Pi$. Also let $Q_i:=P_i\cap P_{i+1}$, for $i=1,\ldots,r$, both as parabolic subgroups and subsets of $\Pi$. Then the Zelevinsky resolution of $X(w)$ is expressed as follows:
\begin{equation}\label{zelres}P_1\times^{Q_1}P_2\times\ldots\times^{Q_{r-2}}P_{r-1}\times^{Q_{r-1}} X(w^r)=:\widetilde{X}(w_\lambda)\rightarrow X(w_\lambda)\,.\end{equation}
Therefore, the cohomology class of $\widetilde{X}(w_\lambda)$ is computed by the following composition of relative push-pull operators:
\begin{equation}\label{push-pull}Y_{P_1/Q_1}\ldots Y_{P_r/Q_r}Y_J\,.\end{equation}
Here we recall the fact that $Y_{J/J'}\bullet \_:(\bfD_t^*)^{W_{J'}}\to (\bfD_t^*)^{W_J}$ coincides with the canonical map $\fh_T(G/P_{J'})\to \fh_T(G/P_J)$; see \cite[Lemma 8.13]{CZZ3} for more details. 

\begin{exa}\label{exzel1}{\rm Continuing Example~\ref{exzel}, the operator in~\eqref{push-pull} is written explicitly as follows:
\[Y_{(\Pi\setminus\{4,6\})/(\Pi\setminus\{4,5,6\})}\,Y_{(\Pi\setminus\{5\})/(\Pi\setminus\{5,7\})}\,Y_{(\Pi\setminus\{7\})/(\Pi\setminus\{5,7\})}\,Y_{\Pi\setminus\{5\}}\,.\]
Indeed, the parabolic subsets $P_i$ for this examples were exhibited in~\cite{BL00}, while they can also be read off from the Young diagram of $\lambda=(5,5,3,2,2)$ as indicated above.
}\end{exa}

We will now state the main technical result of this section, Theorem~\ref{kl-via-gamma}, which is interesting itself, and  is needed to make the connection with the KL-Schubert classes for the Grassmannian, cf.~\cite{LZZ}. To this end, we introduce more notation in the above setup. Given the rectangle $R_i$, with its embedding into the Young diagram of $\lambda$ and the first quadrant,  let $C_i$ and $D_i$ be the sets of labels on its left vertical side and its top horizontal side, respectively. Let
\[c_i:=\min\,C_i\,,\;\;\:d_i:=\max\,D_i=c_i+p_i+q_i-1\,,\;\;\:C'_i:=C_i\setminus\{\max\,C_i\}\,,\;\;\:D'_i:=D_i\setminus\{d_i\}\,.\] 
Finally, let $J_i:=C_i\sqcup D'_i$ and $J'_i:=C'_i\sqcup D'_i$. 

We also need to define the subsets $K'_i\subsetneq K_i$ of $\Pi$, $i=1,\ldots,r$. First recall that above we defined the shape $\rho^i$ as the union of the rectangles $R_1,\ldots,R_i$. It is not hard to see that $\rho^i$ is a union of completely disjoint Young diagrams (i.e., they do not share even a single point), aligned from southwest to northeast. Let $\mathcal{C}_i$ be set of indices $j\in\{1,\ldots,i\}$ such that the left side of $R_j$ is contained in the left boundary of a component of $\rho^i$. Similarly, let $\mathcal{D}_i$ be set of indices $k\in\{1,\ldots,i\}$ such that the top side of $R_k$ is contained in the top boundary of a component of $\rho^i$. We now define
\[K'_i:=\left(\bigsqcup_{j\in\mathcal{C}_i} C'_j\right)\sqcup\left(\bigsqcup_{k\in\mathcal{D}_i} D'_k\right)\,,\;\;\;\;\;K_i:=K'_i\sqcup\{\max\,C_i\}\,.\]
Note that $J_i\subseteq K_i$ and $J_i'\subseteq K_i'$.

\begin{exa}\label{exzel3}{\rm Continuing Example~\ref{exzel1}, we have
\begin{align*}&K'_1=J'_1=\emptyset\subsetneq K_1=J_1=\{5\}\,,\;\;\;\;K'_2=J'_2=\{6,8,9\}\subsetneq K_2=J_2=\{6,7,8,9\}\,,\\
&J_3'=\{1,2,3,4,6\}\subsetneq J_3=\{1,2,3,4,5,6\}\,,\;\;\;\;K'_3=\{1,2,3,4,6,8,9\}\subsetneq K_3=\{1,2,3,4,5,6,8,9\}\,.\end{align*}
As indicated above, all this information is easily read off from the Young diagram of $\lambda=(5,5,3,2,2)$.
}
\end{exa}

\begin{thm}\label{kl-via-gamma} In $H\subset Q_{m,W} $, we have
\begin{equation}\label{X-gamma}\gamma_{w_\lambda w_J}=\gamma_{J_1/J'_1}\ldots\gamma_{J_r/J'_r}\gamma_J=\gamma_{K_1/K'_1}\ldots\gamma_{K_r/K'_r}\gamma_J\,.\end{equation}
\end{thm}

In order to prove Theorem~\ref{kl-via-gamma}, we start by recalling some results from~\cite{KL}, related to the factorization of Kazhdan-Lusztig elements for the Grassmannian. 
This paper introduces an element  $Z_{w_\lambda}$ of the Hecke algebra, defined as a product of linear factors in the generators, which are associated with the boxes of the Young diagram $\lambda$. Instead of recalling the precise definition, which is not needed here, we will state a weaker form of the factorization, which turns out to be related to factorizations in~\eqref{X-gamma}. We will use notation introduced above.

The rectangle $R_i$ corresponds to the following Grassmannian permutation, cf.~\eqref{fact} and Example~\ref{exlam}:
\[v^i:=(s_{c_i+q_i-1}\ldots s_{c_i})(s_{c_i+q_i}\ldots s_{c_i+1})\ldots(s_{c_i+p_i+q_i-2}\ldots s_{c_i+p_i-1})\,.\]
It is not hard to see that we have the following factorization of $w_\lambda$, which corresponds to a reduced decomposition of $w_\lambda$ obtained from~\eqref{fact} only by commuting simple reflections:
\begin{equation}\label{newred}w_\lambda=v^1\ldots v^r\,.\end{equation}

\begin{exa}{\rm In our running example, the reduced decomposition corresponding to~\eqref{newred} (to be compared with~\eqref{exfact}, cf. also~\eqref{contents}) is
\[ w_\lambda=[3,4,6,9,10,1,2,5,7,8]=\underbrace{(s_5)}_{v^1}\underbrace{((s_8s_7s_6)(s_9s_8s_7))}_{v^2}\underbrace{((s_2s_1)(s_3s_2)(s_4s_3)(s_5s_4)(s_6s_5))}_{v^3}\,.\]
}\end{exa}

The factorization of $Z_{w_\lambda}$ needed here is the following one, which corresponds to the factorization~\eqref{newred} of $w_\lambda$:
\begin{equation}\label{factz}Z_{w_\lambda}=Z_{v^1}Z_{w^2}=Z_{v^1}\ldots Z_{v^r}\,.\end{equation}
See the proof of \cite[Theorem~3]{KL} for details.

The connection between the element $Z_{w_\lambda}$ and the corresponding parabolic Kazhdan-Lusztig basis element is made in \cite[Theorem~3]{KL}. 

\begin{thm}\label{res-kl}\cite{KL} In $H\subset Q_{m,W}$, we have
\[Z_{w_\lambda}\gamma_J=\gamma_{w_\lambda w_J}\,.\]
\end{thm}

The proof of Theorem~\ref{kl-via-gamma} also relies on the following lemmas.

\begin{lem}\label{PQ} Consider $J'\subset J\subseteq \Pi$, and assume that $J\subset [a,b]$ with $a,b\in \Pi$. If $A\subseteq \Pi\setminus[a-1,b+1]$, then we have 
\begin{equation*}\gamma_{J/J'}=\gamma_{J\sqcup A/J'\sqcup A}\in Q_{m,W} \,,\;\;\;\;\;Y_{J/J'}=Y_{J\sqcup A/J'\sqcup A}\in \bfD_t\,.
\end{equation*}
\end{lem}

\begin{proof}
As the sets of simple roots corresponding to $J$ and $A$ are orthogonal to each other, we have $\Sigma_{J\sqcup A}^-=\Sigma_J^-\sqcup\Sigma_A^-$, $W_{J\sqcup A}=W_J\times W_A$, and similarly for $J$ replaced by $J'$. Therefore, we have
\begin{equation}\label{wjjp}w_{J/J'}:=w_Jw_{J'}=w_Jw_Aw_{J'}w_A=:w_{J\sqcup A/J'\sqcup A}\,,\;\;\;\;\;x_{J/J'}=x_{J\sqcup A/J'\sqcup A}\,,\end{equation}
and $W_J/W_{J'}$ is in a natural bijection with $W_{J\sqcup A}/W_{J'\sqcup A}$. The stated equalities follow by plugging these facts into~\eqref{defgjjp} and the definition~\eqref{eq:relpush} of the relative push-pull operator. 
\end{proof}

Denote $\hat x_{\Pi}:=\prod_{\al<0}(t-t^{-1}e^{\al})$. We define an anti-involution 
\[\hiota:Q_{m, W}\to Q_{m,W}\,, \quad  \hiota(z_1z_2):=\hiota(z_2)\hiota(z_1) \;\,\text{ and }\;\,\hiota(p\de^m_w):=\de^m_{w^{-1}}p\frac{w(\hat x_{\Pi}x_{\Pi})}{\hat x_{\Pi}x_\Pi}\,. \]
Note that this is different from the anti-involution $\iota$ in \eqref{eq:inv1} or the involution $\overline{\{\cdot\}}$ in \eqref{eq:inv2}. 

\begin{lem}\label{use-i} 
{\rm (1)} In $Q_{m,W}$, we have $\hiota(\tau_i)=\tau_i$ and $\hiota(\gamma_w)=\gamma_{w^{-1}}$.

{\rm (2)} Given $J'\subset J\subseteq \Pi$, we have 
\begin{equation*}\gamma_{J}=\gamma_{J'}\,\hiota(\gamma_{J/J'})\in Q_{m,W}\,,\;\;\;\;\;Y_{J}=Y_{J'}\,\iota(Y_{J/J'})\in \bfD_t\,.\end{equation*}
\end{lem}

\begin{proof}
(1) The first property follows from direct computation, and the second one follows from the fact that $P_{v,w}=P_{v^{-1}, w^{-1}}$. 

(2) By the first part, we have $\hiota(\gamma_J)=\gamma_J$. Based on this fact and \eqref{mult-gamma}, we have 
\[\gamma_J=\hiota(\gamma_J)=\hiota(\gamma_{J/J'}\,\gamma_{J'})=\hiota(\gamma_{J'})\,\hiota(\gamma_{J/J'})=\gamma_{J'}\,\hiota(\gamma_{J/J'})\,.\]
The similar property for $Y_J$ follows in the same way from~\eqref{invar-yj} and~\eqref{mult-y}.
\end{proof}

\begin{lem}\label{j-to-k} {\rm (1)} We have
\[K_1=J_1\supsetneq K'_1=J'_1\subsetneq K_2\supsetneq K'_2\subsetneq\ldots\subsetneq K_r\supsetneq K'_r\subseteq J\,.\]

{\rm (2)} For every $i=1,\ldots,r$, we have
\begin{equation*}\gamma_{J_i/J'_i}=\gamma_{K_i/K'_i}\in Q_{m,W}\,,\;\;\;\;\;Y_{J_i/J'_i}=Y_{K_i/K'_i}\in \bfD_t\,.\end{equation*}
\end{lem}
\begin{proof} It is clear that $K'_r\subseteq J$. Thus, in order to complete the first part, we need to prove $K'_{i-1}\subsetneq K_i$, for $i=2,\ldots,r$. This is obvious if the rectangle $R_i$ is, by itself, a connected component of the shape $\rho^i$. 
Other than this, there are three ways in which $R_i$ can be attached to $\rho^{i-1}$, which are indicated below; the boxes of $R_i$ are marked with $\star$, and the empty boxes form the relevant component(s) of $\rho^{i-1}$. 
\begin{equation}\label{attach-R}\tableau{\star&\star&{}&{}&{}&{}\\\star&\star&{}&{}&{}\\\star&\star&{}\\\star&\star\\\star&\star}\qquad\;\;
\tableau{\star&\star&\star&\star\\\star&\star&\star&\star\\{}&{}&{}\\{}&{}\\{}}\;\;\qquad
\tableau{\star&\star&\star&\star&{}&{}\\\star&\star&\star&\star&{}\\\star&\star&\star&\star\\{}&{}&{}\\{}&{}}\end{equation}
Note that the height (respectively width) of $R_i$ is strictly greater than the number of rows (respectively columns) of the relevant Young diagram to its right (respectively at the bottom).  It is also useful to observe that all unit segments with the same label form a northwest to southeast staircase shape, and the labels increase by 1 as we move northeast. 

Let $B$ denote the set of labels on the boundary of the rectangle $R_i$. 
Using the above notation, in all three cases in~\eqref{attach-R}, we have 
\[B=C_i\sqcup D_i=\{c_i,\ldots,d_i\}\,,\;\;\;K_i\setminus B=K'_{i-1}\setminus B,\,\;\;\;K_i\cap B=C_i\sqcup D'_i=B\setminus\{d_i\}=:J_i\,.\]
On another hand, we have $d_i\not\in K'_{i-1}$; indeed, in the first and last case in~\eqref{attach-R}, the label $d_i$ is on the left side of a rectangle $R_j$ with $j\in{\mathcal C}_{i-1}$, but $d_i\not\in C_j'$, because it is the top label on the mentioned side. We conclude that $K'_{i-1}\subseteq K_i$. In fact, the inclusion is strict because we also have $c_i+q_i-1\in (K_i\cap B)\setminus K'_{i-1}$.

For the second part, we note that, in addition to the above facts, we have $K'_i\cap B=C'_i\sqcup D'_i=:J'_i$ and $c_i-1\not \in K_i$. For the latter part, note that, in the last two cases in~\eqref{attach-R}, the label $c_i-1$ is on the left side of a rectangle $R_j$ with $j\in{\mathcal C}_{i}$ and $j\ne i$, but $c_i-1\not\in C_j'$, because it is the top label on the mentioned side. The proof is concluded by applying Lemma~\ref{PQ}.
\end{proof}

\begin{proof}[Proof of Theorem~{\rm \ref{kl-via-gamma}}]
Using the above setup, we have
\begin{align}\label{use-i-rep}
\gamma_{K_2/K'_2}\ldots\gamma_{K_r/K'_r}\gamma_J&\,\overset{\sharp_1}=\,\gamma_{K_2/K'_2}\ldots\gamma_{K_r/K'_r}\gamma_{K'_r}\,\hiota(\gamma_{J/K'_r})\\
&\,\overset{\sharp_2}=\,\gamma_{K_2/K'_2}\ldots\gamma_{K_r}\,\hiota(\gamma_{J/K'_r})\nonumber\\
&\,\overset{\sharp_3}=\,\gamma_{K_2/K'_2}\ldots\gamma_{K'_{r-1}}\,\hiota(\gamma_{K_r/K'_{r-1}})\,\hiota(\gamma_{J/K'_r})\nonumber\\
&\,=\,\ldots\,\overset{\sharp_4}=\,\gamma_{K'_1}\,\hiota(\gamma_{K_2/K'_1})\ldots\hiota(\gamma_{K_r/K'_{r-1}})\,\hiota(\gamma_{J/K'_r})\,.\nonumber
\end{align}
Here $\sharp_1,\,\sharp_3$ are based on Lemma~\ref{j-to-k}~(1) and Lemma~\ref{use-i}~(2), $\sharp_2$ on~\eqref{mult-gamma}, and $\sharp_4$ on the repeated use of an argument similar to $\sharp_2$ followed by $\sharp_3$. 

We now prove the theorem using induction on $r$, with base case $r=0$, which is trivial. We have
\begin{align*}
\gamma_{w_\lambda w_J}&\,\overset{\sharp_1}=\,Z_{w_\lambda}\gamma_J\overset{\sharp_2}=Z_{v^1}Z_{w^2}\gamma_J\overset{\sharp_3}=Z_{v^1}\gamma_{w^2w_J}\\
&\,\overset{\sharp_4}=\,Z_{v^1}\,\gamma_{J_2/J'_2}\ldots\gamma_{J_r/J'_r}\gamma_J\,\overset{\sharp_5}=\,Z_{v^1}\, \gamma_{K_2/K'_2}\ldots\gamma_{K_r/K'_r}\gamma_J\\
&\,\overset{\sharp_6}=\,Z_{v^1}\, \gamma_{K'_1}\,\hiota(\gamma_{K_2/K'_1})\ldots\hiota(\gamma_{K_r/K'_{r-1}})\,\hiota(\gamma_{J/K'_r})\\
&\,\overset{\sharp_7}=\,\gamma_{K_1}\,\hiota(\gamma_{K_2/K'_1})\ldots\hiota(\gamma_{K_r/K'_{r-1}})\,\hiota(\gamma_{J/K'_r})\\
&\,\overset{\sharp_8}=\,\gamma_{K_1/K'_1}\gamma_{K'_1}\,\hiota(\gamma_{K_2/K'_1})\ldots\hiota(\gamma_{K_r/K'_{r-1}})\,\hiota(\gamma_{J/K'_r})\\
&\,\overset{\sharp_9}=\,\gamma_{K_1/K'_1}\gamma_{K_2/K'_2}\ldots\gamma_{K_r/K'_r}\gamma_J\,\overset{\sharp_{10}}=\,\gamma_{J_1/J'_1}\gamma_{J_2/J'_2}\ldots\gamma_{J_r/J'_r}\gamma_J\,.
\end{align*}
Here $\sharp_1,\,\sharp_3,\,\sharp_7$ are based on Theorem~\ref{res-kl}, $\sharp_2$ on~\eqref{factz}, $\sharp_4$ on the induction hypothesis, $\sharp_5,\,\sharp_{10}$ on Lemma~\ref{j-to-k}~(2), $\sharp_6,\,\sharp_{9}$ on~\eqref{use-i-rep}, and $\sharp_8$ on~\eqref{mult-gamma}; additionally, in $\sharp_7$ we use the fact that
\[K_1=J_1=C_1\sqcup D'_1=\{c_1,\ldots,d_1-1\}\,,\;\;\;\;{K}_1'={J}_1'={C}_1'\sqcup D'_1=K_1\setminus\{\max\,C_1\}\,,\]
and thus we have $v^1 w_{K'_1}=w_{K_1}$.
\end{proof}

\begin{rema}{\rm We could not have carried out the above proof by using only one of the pairs $(J_i,J_i')$ and $(K_i,K_i')$. Indeed, the first pair does not satisfy the property in Lemma~\ref{j-to-k}~(1), which is crucial in the proof. On the other hand, the induction procedure cannot be applied based on the second pair because the respective sets for $\lambda^1=\lambda$ and $\lambda^2$ (corresponding to $w^2$) are different.
}\end{rema}

In order to relate Theorem~\ref{kl-via-gamma} to the Zelevinsky resolution, and more specifically to the operator~\eqref{push-pull}, we need the following result.

\begin{lem}\label{JPQ} For every $i=1,\ldots,r$, we have
\[Y_{J_i/J'_i}=Y_{K_i/K'_i}=Y_{P_i/Q_i}\,.\]
\end{lem}

\begin{proof} By using Lemma~\ref{j-to-k}~(2), it suffices to prove $Y_{J_i/J'_i}=Y_{P_i/Q_i}$. Moreover, it suffices to consider  $i=1$, as we can just replace the partition $\lambda^1=\lambda$ with $\lambda^i$. Recall that $P_1$ is obtained by considering the lattice path from $(0,0)$ to $(n-d,d)$ defining the southeast boundary of $\lambda^1$, and by excluding from $\Pi$ the last label in each sequence of vertical steps. Similarly, $P_2$ corresponds to $\lambda^2:=\lambda\setminus R_1$. 

Let $B$ denote the set of labels on the boundary of the rectangle $R_1$; see the diagram below, where the boxes of $R_1$ are marked with $\star$.
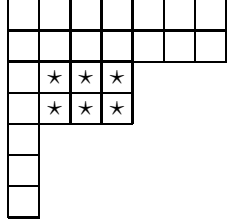
\[\tableau{{}&{}&{}&{}&{}&{}&{}\\{}&{}&{}&{}&{}&{}&{}\\{}&\star&\star&\star\\{}&\star&\star&\star\\{}\\{}\\{}}\]
 Using the above notation, we have $B=C_1\sqcup D_1=\{c_1,\ldots,d_1\}$. Based on the above interpretation of $P_1$ and $P_2$, we deduce
\begin{align*}
&P_1\cap B=C_1\sqcup D'_1=:J_1=B\setminus\{d_1\},\;\;\;\;P_2\cap B={C}_1'\sqcup D_1\;\;\Longrightarrow\;\;Q_1\cap B=C'_1\sqcup D'_1=:J'_1,\\
&P_1\setminus B\subset P_2\setminus B\;\;\Longrightarrow\;\;P_1\setminus B=Q_1\setminus B\,.
\end{align*}
Moreover, we have $c_1-1\not\in P_1$ and $d_1\not\in P_1$. Thus, we are under the hypotheses of Lemma~\ref{PQ}, so the conclusion follows.
\end{proof}

We now rephrase Theorem~\ref{kl-via-gamma} as follows, via the map $\psi$.

\begin{cor}\label{kl-via-gamma-new}
We have
\begin{equation}\label{kl-fact}\mu^{-\ell(w_\lambda w_J)}\psi(\gamma_{w_\lambda w_J})=Y_{P_1/Q_1}\ldots Y_{P_r/Q_r}Y_J\in \bfD_t\,.\end{equation}
\end{cor}

\begin{proof} We start by observing the following:
\begin{equation}\label{wkkp}
w_{K_i/K_i'}=w_{J_i/J_i'}=v^i\;\;\;\;\Longrightarrow\;\;\;\;\ell(w_{K_i/K_i'})=p_iq_i=|R_i|\,,
\end{equation}
where $|R_i|$ denotes the number of boxes of the rectangle $R_i$. Here the first equality is based on~\eqref{wjjp} and the fact that this result can be applied to the pairs $(J_i,J_i')$ and $(K_i,K_i')$, as discussed in the proof of Lemma~\ref{j-to-k}; the second equality is clear by the definition of $v^i$.

We now apply $\mu^{-\ell(w_\lambda w_J)}\psi(\,\cdot\,)$ to the first and last part of~\eqref{X-gamma}. After doing this, the latter can be written as follows:
\begin{align*}
\;\;\;\;\;\;\;&\;\;\;\!\mu^{-\ell(w_\lambda w_J)}\psi(\gamma_{K_1/K'_1})\ldots\psi(\gamma_{K_r/K'_r})\psi(\gamma_{J})\\
\overset{\sharp_1}=&\left(\mu^{-\ell(w_{K_1/K_1'})}\psi(\gamma_{K_1/K'_1})\right)\ldots\left(\mu^{-\ell(w_{K_r/K_r'})}\psi(\gamma_{K_r/K'_r})\right)\left(\mu^{-\ell(w_J)}\psi(\gamma_{J})\right)\\
\overset{\sharp_2}=&\left(\mu^{-\ell(w_{K_1/K_1'})}\psi(\gamma_{K_1/K'_1})\right)\ldots\left(\mu^{-\ell(w_{K_r/K_r'})}\psi(\gamma_{K_r/K'_r})\right)Y_J\\
\overset{\sharp_3}=&\left(\mu^{-\ell(w_{K_1/K_1'})}\psi(\gamma_{K_1/K'_1})\right)\ldots\left(\mu^{-\ell(w_{K_r/K_r'})}\psi(\gamma_{K_r/K'_r})\right)Y_{K_r'}\iota(Y_{J/K_r'})\\
\overset{\sharp_4}=&\left(\mu^{-\ell(w_{K_1/K_1'})}\psi(\gamma_{K_1/K'_1})\right)\ldots Y_{K_r}\iota(Y_{J/K_r'})\\
=&\ldots\overset{\sharp_5}=Y_{K_1}\iota(Y_{K_2/K_1'})\ldots\iota(Y_{J/K_r'})\\
\overset{\sharp_6}=\,&\,Y_{K_1/K_1'}Y_{K_1'}\iota(Y_{K_2/K_1'})\ldots\iota(Y_{J/K_r'})\overset{\sharp_7}=Y_{K_1/K_1'}Y_{K_2}\ldots\iota(Y_{J/K_r'})\\
=&\ldots\overset{\sharp_8}=Y_{K_1/K_1'}\ldots Y_{K_r/K_r'}Y_J\overset{\sharp_9}=Y_{P_1/Q_1}\ldots Y_{P_r/Q_r}Y_J\,.
\end{align*}
Here $\sharp_1$ is based on~\eqref{wkkp} and the fact that $\ell(w_\lambda)=\sum_i |R_i|$, $\sharp_2,\,\sharp_4$ are based on~\eqref{eq:gammapsirel}, $\sharp_3,\,\sharp_7$ on Lemma~\ref{use-i}~(2), $\sharp_5$ on the repeated use of an argument similar to $\sharp_3$ followed by $\sharp_4$, $\sharp_6$ on~\eqref{mult-y}, $\sharp_8$ on the repeated use of an argument similar to $\sharp_6$ followed by $\sharp_7$, and $\sharp_9$ on Lemma~\ref{JPQ}.
\end{proof}

We now state the main result of this section.

\begin{thm}\label{kl-zel}
The KL-Schubert classes for the Grassmannian coincide with the hyperbolic cohomology classes of the corresponding Zelevinsky resolutions.
\end{thm}

\begin{proof}
{The result is now immediate by comparing the left- and right-hand sides of~\eqref{kl-fact} with Definition~\ref{def:KL} and \eqref{push-pull}, respectively.}
\end{proof}

\begin{rema}{\rm Theorem~\ref{kl-zel} implies that all the Zelevinsky resolutions of a Schubert variety in the Grassmannian have the same class in hyperbolic cohomology (i.e., the corresponding KL-Schubert class). This agrees with a result of Totaro's~\cite{Tot}, which says that the algebraic theories in a larger class (defined by Krichever~\cite{BB}), which includes hyperbolic cohomology, are invariant under small resolutions.}
\end{rema}

\bibliographystyle{alpha}

\end{document}